\documentclass[11pt,twoside,reqno]{amsart}
\usepackage{microtype}
\usepackage[OT1]{fontenc}
\usepackage{type1cm}
\usepackage{amssymb}
\usepackage[dvips]{graphicx}
\usepackage[bf]{caption}
\usepackage{psfrag}          % replace text in figures
\usepackage[left=2.3cm,top=3cm,right=2.3cm]{geometry}        % for a good layout
\usepackage{version}         % include, exclude and comment
\usepackage[active]{srcltx}  % for the inverse search
\usepackage[colorlinks = true, citecolor = black, linkcolor = black, urlcolor = black, pdfstartview=FitH]{hyperref} % links
\usepackage[11pt]{moresize}
\usepackage{textcomp}
\numberwithin{equation}{section}

\theoremstyle{plain}
\newtheorem{theorem}{Theorem}[section]
\newtheorem{corollary}[theorem]{Corollary}
\newtheorem{proposition}[theorem]{Proposition}
\newtheorem{lemma}[theorem]{Lemma}

\theoremstyle{remark}
\newtheorem{remark}[theorem]{Remark}

\theoremstyle{definition}
\newtheorem{definition}[theorem]{Definition}
\newtheorem{notation}[theorem]{Notation}

\newtheorem*{notation*}{Notation}

\renewcommand{\AA}{\mathcal{A}}

\newcommand{\PP}{\mathcal{P}}
\newcommand{\MM}{\mathcal{M}}

\newcommand{\DD}{\mathcal{D}}
\newcommand{\R}{\mathbb{R}}

\newcommand{\Q}{\mathbb{Q}}
\newcommand{\N}{\mathbb{N}}

\newcommand{\cB}{\mathcal{B}}

\newcommand{\cM}{\mathcal{M}}

\newcommand{\cD}{\mathcal{D}}
\newcommand{\e}{\epsilon}
\newcommand{\n}{\mathbf{n}}
\newcommand{\cF}{\mathcal{F}}

\newcommand{\cG}{\mathcal{G}}
\newcommand{\cR}{\mathcal{R}}

\newcommand{\cO}{\mathcal{O}}

\newcommand{\cU}{\mathcal{U}}

\newcommand{\cP}{\mathcal{P}}
\newcommand{\cS}{\mathcal{S}}
\newcommand{\cA}{\mathcal{A}}
\newcommand{\MD}{\mathcal{MD}}
\newcommand{\TD}{\mathcal{TD}}
\newcommand{\FD}{\mathcal{FD}}
\newcommand{\EFD}{\mathcal{EFD}}
\newcommand{\CPD}{\mathcal{CPD}}
\newcommand{\ECPD}{\mathcal{ECPD}}
\newcommand{\eps}{\varepsilon}

\newcommand{\Z}{\mathbb{Z}}

\renewcommand{\epsilon}{\varepsilon}
\renewcommand{\rho}{\varrho}
\renewcommand{\phi}{\varphi}

\newcommand{\ol}{\overline}

\renewcommand{\iint}{\int\hspace{-0.1in}\int}

  \newcommand{\la}{\langle}
  \newcommand{\ra}{\rangle}

\providecommand{\spl}{\mathop{\rm SPL}\nolimits}

\providecommand{\dist}{\mathop{\rm dist}\nolimits}

\providecommand{\spt}{\mathop{\rm spt}\nolimits}

\providecommand{\supp}{\mathop{\rm supp}\nolimits}
\providecommand{\cent}{\mathop{\rm cent}\nolimits}
\providecommand{\Lip}{\mathop{\rm Lip}\nolimits}
\providecommand{\Tan}{\mathop{\rm Tan}\nolimits}

\begin{document}

\title[Distributions generated by the scenery flow]{Structure of distributions generated by the scenery flow}

 \author{Antti K\"aenm\"aki}
 \address{Department of Mathematics and Statistics \\
          P.O.\ Box 35 (MaD) \\
          FI-40014 University of Jyv\"askyl\"a \\
          Finland}
 \email{antti.kaenmaki@jyu.fi}

 \author{Tuomas Sahlsten}
 \address{Einstein Institute of Mathematics\\
	The Hebrew University of Jerusalem\\
         Givat Ram, Jerusalem 91904 \\
         Israel}
\email{tuomas@sahlsten.org}

 \author{Pablo Shmerkin}
 \address{Department of Mathematics and Statistics\\
 	Torcuato Di Tella University\\
	Av. Figueroa Alcorta 7350, Buenos Aires\\ 
	Argentina}
\email{pshmerkin@utdt.edu}

 \thanks{T.S. acknowledges the support from University of Bristol, the Finnish Centre of Excellence in Analysis and Dynamics Research, Emil Aaltonen Foundation and European Union (ERC grant $\sharp$306494). P.S. was partially supported by a Leverhulme Early Career Fellowship and by Project PICT 2011-0436 (ANPCyT)}
%\date{\today}

\maketitle

\begin{abstract}
We expand the ergodic theory developed by Furstenberg and Hochman on dynamical systems that are obtained from magnifications of measures. We prove that any fractal distribution in the sense of Hochman is generated by a uniformly scaling measure, which provides a converse to a regularity theorem on the structure of distributions generated by the scenery flow. We further show that the collection of fractal distributions is closed under the weak topology and, moreover, is a Poulsen simplex, that is, extremal points are dense. We apply these to show that a Baire generic measure is as far as possible from being uniformly scaling: at almost all points, it has all fractal distributions as tangent distributions.
\end{abstract}

%\tableofcontents

\section{Introduction}

\subsection{Historical background}

A central theme in analysis over the years has been the study of ``tangents'' of possibly complicated objects, in order to take advantage of the regularity arising in the limiting structures through the metamorphosis of magnification. For example, a differentiable function looks locally like an affine map, which is more regular than, and gives information about, the original function. In \cite{Preiss1987}, Preiss introduced the more general notion of tangent measure and employed it to solve some outstanding open problems in the theory of rectifiability. Tangent measures are useful because, again, they are more regular than the original measure (for example, tangent measures of rectifiable measures are flat) but one can still pass from information about the tangent measure to the original measure. As another example of the general idea, for certain non-conformal repellers the tangent sets and measures turn out to have a regular product structure which is absent in the more complicated original object; see \cite{BandtKaenmaki2013,FergusonFraserSahlsten2013}. The process of taking blow-ups of a measure or a set around a point in fact induces a natural dynamical system consisting in ``zooming in'' around the point. This opens a door to ergodic-theoretic methods, which were pioneered by Furstenberg in \cite{Furstenberg1970} and then in more developed form in \cite{Furstenberg2008}, with a comprehensive theory developed by Hochman in \cite{Hochman2010}.

In turns out that for some geometric problems, notably those involving some notion of dimension, the ``correct'' class of tangent objects to consider are not tangent measures, but the empirical distributions that appear by magnifying around a typical point. That is, the tangent objects are \emph{measures on measures}, which we call \emph{tangent distributions} (precise definitions will be given in Section \ref{sec:background} below). The reason for this is that tangent measures are defined as weak limits of magnifications around a point, but the sequence along which a tangent measure arises can be very sparse, and for many problems only the behavior on a positive proportion of scales is significant. Tangent distributions are supported on tangent measures which reflect precisely the structure of the original measure on a positive density set of scales.

Furstenberg's key innovation was the introduction of a Markov process on the $b$-adic scaling sceneries of a measure, which he called a \textit{conditional probability (CP) chain}. Since then, CP chains proved to be a key tool to solve several important problems in fractal geometry, probability theory and ergodic theory. In \cite{Furstenberg2008}, Furstenberg applied this technology to understand dimension conservation of homogeneous measures. Then Hochman and Shmerkin used them to study projections of fractal measures \cite{HochmanShmerkin2009} and the behavior of measures with respect to normal numbers \cite{HochmanShmerkin2013}. Furthermore, recently Orponen \cite{Orponen2012}, and Ferguson, Fraser and Sahlsten \cite{FergusonFraserSahlsten2013} found connections to the distance set conjecture for several dynamically defined sets. However, CP chains are defined in a discrete dyadic (or $b$-adic) fashion, and as a result the point that is being zoomed upon is not ``in the center of the frame'' which is often a disadvantage.

An alternative approach is to consider \textit{scenery flows}, in which the magnification is carried out continuously with the point in the center of the frame. Scenery flows were studied (sometimes with this name and sometimes under different names) by many authors, both for specials classes of sets and measures, and in general. We refer to \cite{Hochman2010} for a historical discussion and references. M\"orters and Preiss \cite{MortersPreiss1998} proved the surprising fact that, when dealing with Ahlfors regular measures, the tangent distributions are Palm distributions, which are distributions with a strong degree of symmetry and translation invariance. Hochman \cite{Hochman2010} then showed that a similar phenomenon holds for all Radon measures: he proved that tangent distributions for any measure are almost everywhere quasi-Palm distributions, which is a weaker notion than Palm but still represents a strong spatial invariance. Hochman named distributions which are scale-invariant and enjoy the quasi-Palm property as \emph{fractal distributions}. He also proved the remarkable fact that distributions of CP chains give rise to fractal distributions in a natural way and, reciprocally, any fractal distribution can be obtain from the distribution of a CP chain. The main definitions and results from \cite{Hochman2010} are recalled below, in Section \ref{sec:background}.

\subsection{Summary of main results}

In this work, we continue developing the theory of CP chains and fractal distributions. We state our main results in somewhat informal fashion; precise definitions and statements are postponed to the later sections.

Since fractal distributions are the cornerstone of the theory developed by Hochman in \cite{Hochman2010}, a natural problem is to study the topological structure of the family of fractal distributions.
\begin{theorem} \label{thm:FDs-are-closed}
The family of fractal distributions is closed with respect to the weak topology.
\end{theorem}
At first sight this may appear rather surprising, since the scaling flow is not continuous, its support is not closed and, more significantly, the quasi-Palm property is not a closed property. Thus, this theorem is another manifestation of the general principle that, although fractal distributions are defined in terms of seemingly strong and discontinuous geometric properties, they are in fact very robust. Besides its intrinsic interest, Theorem \ref{thm:FDs-are-closed} has a number of applications in classical problems in the realm of geometric measure theory, which we develop elsewhere; see \cite{KSS2013}. Indeed, these applications were our initial motivation to continue developing the general theory of fractal distributions.

Recall that a \emph{Choquet simplex} $\Delta$ in a locally convex topological vector space is a compact convex set with the property that each $x\in \Delta$ can by expressed in a unique way (up to measure zero sets) as an integral $\int y \, \mathrm{d} P(y)$ for some probability distribution $P$ on the extremal points of $\Delta$. It follows from Theorem \ref{thm:FDs-are-closed} and results in \cite{Hochman2010} that the family of fractal distributions is in fact a Choquet simplex, so another question arises: what kind of Choquet simplex is it?

\begin{theorem} \label{thm:Poulsen}
The family of fractal distributions is a Poulsen simplex.
\end{theorem}

A \emph{Poulsen simplex} is a non-trivial Choquet simplex in which extremal points are dense. A classical result of Lindenstrauss, Olsen, and Sternfeld \cite{LOS1978} states that there is in fact a unique Poulsen simplex up to affine homeomorphism. In that paper two other striking properties of the Poulsen simplex are established: any affine homeomorphism between two proper faces of the Poulsen simplex may be extended to a homeomorphism of the whole simplex (homogeneity), and any metrizable simplex is affinely homeomorphic to a face of the Poulsen simplex. The Poulsen simplex is a common object in ergodic theory, as the space of invariant measures for many dynamical systems is Poulsen; this is often a manifestation of some kind of hyperbolic behavior.

In our case, the set of extremal points is precisely the collection of ergodic fractal distributions with respect to the scenery flow. We remark
that the dense set we exhibit consists of distributions of random self-similar measures, where the self-similarity is with respect to a $b$-adic grid. This potentially allows to prove certain statements for arbitrary measures or distributions by reducing it to this fairly concrete and well-behaved class.
The construction of these self-similar measures is a special case of what we term the \emph{splicing} of scales. Roughly speaking, this consists in pasting together a sequence of measures along dyadic scales; see Section \ref{subsec:splicing} for more details. Splicing is often employed to construct sets or measures with a given property based on properties of the component measures. For example, in \cite{SchmelingShmerkin2010}, splicing was used to investigate the dimensions of iterated sums of a Cantor set, and Hochman \cite[Section 8.3]{Hochman2010} employed it (under the name of \textit{discretization}) to construct examples of uniformly scaling measures with non-ergodic limit geometry and bad projection properties.

Given a measure $\mu$, we can study its geometric properties via its tangent distributions. The situation is especially nice when at $\mu$ almost all points there is a single tangent distribution, and even nicer when all these tangent distributions coincide. This leads us to the concepts of \textit{uniformly scaling measures} (USMs) and \textit{generated} distributions; these concepts were first defined by Gavish \cite{Gavish2011} and investigated further by Hochman \cite{Hochman2010}; see Section \ref{sec:normscenery} for more details.

Uniformly scaling measures are geometrically much more regular than arbitrary measures, for example in the behavior of their projections \cite{Hochman2010,HochmanShmerkin2009} and the distance sets of their supports \cite{FergusonFraserSahlsten2013}. Examples of USMs are many conformal and non-conformal constructions, both deterministic and random \cite{FergusonFraserSahlsten2013,Gavish2011,Hochman2010}, measures invariant under $x\to px\bmod 1$ on the circle \cite{Hochman2012}, and the occupation measure of Brownian motion in dimension $d\ge 3$ \cite{Gavish2011}.

It seems natural to ask what kind of distributions can arise as the (unique) distribution generated by a USM. Hochman \cite{Hochman2010} proved the striking fact that generated distributions are always fractal distributions. We provide a converse to this:

\begin{theorem} \label{thm:measure-generating-FD}
Every fractal distribution is generated by some uniformly scaling measure.
\end{theorem}

Again, our motivation for this result arose from our applications to problems in geometric measure theory; see \cite{KSS2013}. Roughly speaking, our approach there is to study families of measures (for example, measures satisfying certain porosity condition) through the family of tangent distributions to those measures at typical points. A key last step is then to pass from the information gleaned on the fractal distribution side back to information about measures -- this is where Theorem \ref{thm:measure-generating-FD} comes in.

Recall that a property is \emph{Baire generic} if it is satisfied everywhere except possibly in a set of first category, that is, a countable union of sets whose closure has empty interior. A recurrent topic in geometric measure theory and analysis is the behavior of Baire generic objects, such as sets, measures, or functions. For example, in recent years many authors have explored the fractal and multifractal behavior of generic Borel measures; see e.g. \cite{Bayart2012,Bayart2013,BuczolichSeuret2010,Sahlsten2012} and references therein. In this context it seems very natural to study the tangent structures of generic measures. O'Neil \cite{ONeil1995} and Sahlsten \cite{Sahlsten2012} proved that a Baire generic measure has \textit{all} Borel measures as tangent measures at almost every point. Even though Theorem \ref{thm:measure-generating-FD} concerns measures which have a single tangent distribution at typical points, perhaps surprisingly, it gives us an application which shows that the exact opposite holds for a generic measure:

\begin{theorem} \label{thm:Baire}
For a Baire generic Radon measure $\mu$ on $\R^d$, the set of tangent distributions is the set of all fractal distributions at $\mu$ almost every $x$.
\end{theorem}

We also obtain an analogous result for CP distributions, see Proposition \ref{prop:genericCP}. These results are in some sense expected, since there is a heuristic principle that says that Baire generic objects behave ``as wildly as possible''.

The rest of the paper is organized as follows. In Section \ref{sec:background} we recall the main elements and results of Hochman's theory. Precise versions of Theorems \ref{thm:FDs-are-closed}, \ref{thm:Poulsen}, \ref{thm:measure-generating-FD}, and \ref{thm:Baire} are stated and proved in Sections \ref{sec:FDs-closed}, \ref{sec:Poulsen}, \ref{sec:USM}, and \ref{sec:Baire-generic}, respectively. In the appendix, we discuss the independence of our results from the chosen norm.

\section{Scenery flow, fractal and CP distributions}
\label{sec:background}

In this section, we recall the main definitions and results from Hochman's work \cite{Hochman2010}, and provide some minor extensions to the theory developed there. We use much of Hochman's notation but we also introduce some new terms, such as tangent distribution and micromeasure distribution.

\begin{notation*}Equip $\R^d$ with the norm $\| x \| = \max_i|x_i|$ and the induced metric. The closed ball centered at $x$ with radius $r>0$ is denoted by $B(x,r)$. In particular, we write $B_1=B(0,1) = [-1,1]^d$. Given a metric space $X$, we denote the family of all Borel probability measures on $X$ by $\PP(X)$, and the family of all Radon measures on $X$ by $\MM(X)$. When $X$ is locally compact, $\MM(X)$ and $\PP(X)$ are endowed with the weak topology. Recall that $\mu_n\to\mu$ weakly if $\int f\,\mathrm{d}\mu_n\to \int f\,\mathrm{d}\mu$ for all continuous functions $f\colon X\to\R$ of compact support.

Whenever we consider convergence in a space of probability measures, it will be implicitly understood that we are considering the weak convergence. When $X=\R^d$, we write $\MM=\MM(\R^d)$. The space $\MM$ is metrizable, complete, and separable. If $X$ is compact, then also $\PP(X)$ is compact.

Following terminology from \cite{Hochman2010}, we refer to elements of $\PP([-1,1]^d)$ or $\MM$ as \emph{measures}, and to elements of $\mathcal{P}(\mathcal{P}([-1,1]^d))$ and $\mathcal{P}(\MM)$ as \emph{distributions}. Measures will be denoted by lowercase Greek letters $\mu,\nu$, etc and distributions by capital letters $P,Q$, etc.  We use the notation $x \sim \mu$ if a point $x$ is chosen randomly according to a measure $\mu$. Moreover, write $\mu \sim\nu$ if the measures $\mu$ and $\nu$ are \textit{equivalent}, that is, they have the same null-sets. If $X$ and $Y$ are metric spaces, $\mu\in\PP(X)$ and $f\colon X\to Y$ is a Borel map, then the \emph{push-down} $f\mu$ is the measure defined via $f\mu(A)=\mu(f^{-1}A)$.
\end{notation*}

\subsection{Ergodic-theoretic preliminaries}

In this article, we make use of many standard definitions and facts from ergodic theory which we briefly recall here for the reader's convenience. Good general references are the books of Einsiedler and Ward \cite{EinsiedlerWard2011} and Walters \cite{Walters1982}.

Let $(X,\mathcal{B},\mu)$ be a probability space. We say that a transformation $T\colon X\to X$ preserves $\mu$ if it is $\mathcal{B}$-measurable and $T\mu=\mu$; the set of all such transformations is a semigroup under composition. A \emph{measure-preserving system} (m.p.s.) is a tuple $(X,\mathcal{B},\mu,T)$ where $(X,\mathcal{B},\mu)$ be a probability space and $T$ is an action of a semigroup by transformations that preserve $\mu$. That is, there is a semigroup $S$ and for each $s\in S$ there is a map $T_s\colon X\to X$ that preserves $\mu$, such that $T_{s+s'}=T_s \circ T_{s'}$. In this article, the underlying semigroup will always be one of $\N$, $\Z$ (in which case we speak of measure-preserving maps,  since the action is determined by $T:=T_1$), or $\R^+, \R$ (in which case we speak of \textit{flows}). Moreover, for us $X$ will always be a metric space and $\mathcal{B}$ will be the Borel $\sigma$-algebra on $X$ (thus no explicit reference will be made to it). In the following we always assume that we are in this setting to avoid unnecessarily technical assumptions.

A measure-preserving system is \emph{ergodic} if any set $A \in \mathcal{B}$ with $\mu(T_s^{-1}A \triangle A)=0$ for all $s$ has either zero of full $\mu$-measure. For a given action $T$ on a space $X$, ergodic measures are the extremal points of the convex set of all probability measures which are preserved by $T$.  The \emph{ergodic theorem} for discrete actions says that if $f\in L^1(\mu)$ and the system is ergodic, then
\[
 \lim_{n\to\infty} \frac1n \sum_{i=0}^{n-1} f(T^i x) = \int f\, \mathrm{d}\mu \quad\textrm{for }\mu\textrm{ almost all } x.
\]
For flows, the same holds replacing the left-hand side by $\lim_{t \to \infty} \frac1t \int_0^t f \circ T_s \,\mathrm{d}s$.

Given two measure-preserving systems $(X,\mu,T)$ and $(X',\mu',T')$, a map $\pi\colon X\to X'$ is called a \emph{factor map} if the underlying semigroups coincide, $\pi\mu=\mu'$ and $\pi$ intertwines the actions of the semigroups: $\pi T_s = T'_s \pi$ for all $s$. In this case we also say that $(X',\mu',T')$ is a factor of $(X,\mu,T)$ . The factor of an ergodic system is ergodic. When $\pi$ is a measure-theoretical isomorphism, we say that the systems $(X,\mu,T)$ and $(X',\mu',T')$ are \emph{isomorphic}.

Let $(X,\mu)$ be a metric probability space, and consider the product spaces $(X^\N,\mu^\N)$ and $(X^\Z,\mu^\Z)$. The shift map $T$ defined by $T((x_i)_i)=(x_{i+1})_i$ acts on both spaces. This maps preserves the product measure and the resulting system is always ergodic. In general, there may be many other measures on $X^\N$ or $X^\Z$ that are preserved by the shift. If $(X^\N,\mu,T)$ is a m.p.s., there is always a m.p.s.\ $(X^\Z,\widehat{\mu},T)$ such that the former is a factor of the latter under the natural projection map (on the natural extension the $\sigma$-algebra is not the Borel $\sigma$-algebra but the smallest $\sigma$-algebra that makes the projection $x\mapsto x_0$ measurable, but this exception to our convention will cause no trouble). This is called the \emph{two-sided extension} or natural extension for shift spaces. The two-sided extension is ergodic if and only if the one-sided version is ergodic. Furthermore, any discrete m.p.s.\ $(X,\mu,T)$ can always be represented as a shift space via the identification $x\to (T^i x)_{i=0}^\infty$ (that is, the measure on $X^\N$ is the push-down of $\mu$ under this map).

A general m.p.s.\ $(X,\mathcal{B},\mu,T)$ can be decomposed into (possibly uncountably many) ergodic parts according to the \emph{ergodic decomposition} theorem: there exists a Borel map $x\to \mu_x$ from $X$ to $\mathcal{P}(X)$ such that for $\mu$ almost all $x$ it holds that each $(X,\mathcal{B},\mu_x,T)$ is measure-preserving and ergodic, and $\mu = \int \mu_x \,\mathrm{d}\mu(x)$. Moreover, this map is unique up to sets of zero $\mu$-measure. The measures $\mu_x$ are called the \emph{ergodic components} of $\mu$.

A standard way to build measure-preserving flows from discrete systems is via \emph{suspensions}; we only consider the case of a constant roof function. Let $(X,\mu,T)$ be a discrete m.p.s.\ (the base) and let $r>0$ (the height). Write $\widehat{X}=X\times [0,r)$, $\widehat{\mu}=\mu\times\mathcal{L}$, where $\mathcal{L}$ is normalized Lebesgue measure on $[0,r)$, and set $\widehat{T}_t(x,s)=(x,s+t)$ if $s+t<r$ and $\widehat{T}_t(x,r-t)=(Tx,0)$. By iterating this defines a flow (called \emph{suspension flow}) for all $t>0$, which indeed preserves the measure $\widehat{\mu}$. Moreover, the m.p.s.\ $(X,\mu,T)$ is ergodic if and only if $(\widehat{X},\widehat{\mu},\widehat{T})$ is ergodic.

\subsection{Normalizations and the scenery flow}
\label{sec:normscenery}

 If $\mu\in\MM$ and $\mu(A)>0$, then $\mu|_A$ is the restriction of $\mu$ to $A$ and, provided also $\mu(A)<\infty$, we denote by $\mu_A$ the restriction normalized to be a probability measure, that is
\[
\mu_A(B) =\frac{1}{\mu(A)}\mu(A\cap B).
\]
For any $\mu\in\MM$ for which $\mu(B_1)>0$ we define the normalization operations $*,\square$ in $\cM$ by
\begin{eqnarray*}
\mu^* &:=& \frac{1}{\mu(B_1)}\mu,\\
\mu^\square &:=& \mu_{B_1} = \mu^*|_{B_1}.
\end{eqnarray*}
We define the translation and scaling actions on measures by
\begin{eqnarray*}
T_x\mu(A) &=& \mu(A-x),\\
S_t\mu(A) &=& \mu(e^{-t}A).
\end{eqnarray*}
The reason for the exponential scaling is to make $S_t$ into a partial action of $\R$ into $\MM$. Whenever $R$ is an operator on $\MM$, we write $R^*,R^\square$ for the corresponding operator obtained by post-composition with the respective normalizations. So, for example, $T_x^\square\mu=(T_x\mu)^\square$. We also write
$$\MM^*=\{\mu\in\MM:0\in\spt\mu\}$$
and $\MM^\square=\PP(B_1)$.

We note that the actions $S_t^*$ and $S_t^\square$ are discontinuous and fully defined only on the (Borel but not closed) set $\MM^*$. Nevertheless, the philosophy behind many of the results in \cite{Hochman2010} is that in practice they behave in a very similar way to a continuous action on a complete metric space (compact in the case of $S_t^\square$).

\begin{definition}[Scenery flow and tangent measures] We call the flow $S_t^\square$ acting on $\MM^*$ the \textit{scenery flow} at $0$. Given $\mu\in\MM$ and $x\in\spt\mu$, we consider the one-parameter family
\[
\mu_{x,t} := \mu_{x,t}^\square = S_t^\square(T_x\mu)
\]
generated by the action of $S_t^\square$ on $T_x \mu$ and call it the \emph{scenery of $\mu$ at $x$}. Accumulation points of this scenery will be called \emph{tangent measures} of $\mu$ at $x$ and the family of tangent measures of $\mu$ at $x$ is denoted by $\Tan(\mu,x)$.
\end{definition}

\begin{remark}
We deviate slightly from the usual definition of tangent measures, which corresponds to taking weak limits of $S_t^*(T_x\mu)$ instead, i.e.\ without restricting the measures.
\end{remark}

One of the main ideas of \cite{Hochman2010}, which we further pursue in \cite{KSS2013}, is that, as far as certain properties of a measure are concerned (including their dimensions), the ``correct'' tangent structure to consider is not a single limit of $\mu_{x,t_k}$ along some subsequence, but the whole statistics of the scenery $\mu_{x,t}$ as $t\to\infty$.

\begin{definition}[Scenery and tangent distributions] \label{def:scenery-and-tangent-distributions}
The \emph{scenery distribution} of $\mu$ up to time $T$ at $x$ is defined by
$$
  \langle \mu \rangle_{x,T} := \frac{1}{T} \int_0^T \delta_{\mu_{x,t}} \,\mathrm{d}t
$$
for all $0 \le T < \infty$. Any weak limit of $\langle \mu \rangle_{x,T}$ in $\cP(\cM^\square)$ for $T \to \infty$ is called a \emph{tangent distribution} of $\mu$ at $x$. The family of tangent distributions of $\mu$ at $x$ will be denoted $\TD(\mu,x)$.
\end{definition}

Here the integration makes sense since we are on a convex subset of a topological linear space. Since the space of distributions $\cP(\cM^\square)$ is compact, $\TD(\mu,x)$ is always a non-empty compact set at $x \in \spt \mu$. Notice also that every $P \in \TD(\mu,x)$ is supported on $\Tan(\mu,x)$.

\begin{definition}[Generated distributions and uniformly scaling measures] \label{def:usm}
We say that a measure $\mu$ \emph{generates} a distribution $P \in \cP(\cM^\square)$ at $x$ if
$$\TD(\mu,x)=\{P\}.$$
Furthermore, $\mu$ \emph{generates $P$} if it generates $P$ at $\mu$ almost every point. In this case, we say that $\mu$ is a \emph{uniformly scaling measure} (USM).
\end{definition}

If $\mu$ is a uniformly scaling measure, then, intuitively, it means that the collection of views $\mu_{x,t}$ will have well-defined statistics as we zoom-in into smaller and smaller neighborhoods of $x$.

\subsection{Fractal distributions}

It has been observed in various forms that tangent measures and distributions have some kind of additional spatial invariance (the simplest form of this is perhaps the well-known fact that ``tangent measures to tangent measures are tangent measures''; see \cite[Theorem 14.16]{Mattila1995}). A very sharp and powerful formulation of this principle was obtained in \cite{Hochman2010}. In order to state it, we need some additional definitions.

\begin{definition}[Fractal distributions]
Let $P\in \PP(\MM)$. We say that the distribution $P$ is:
\begin{enumerate}
\item \emph{scale-invariant} if it is supported on $\MM^*$, and is invariant under the action of the semigroup $S_t^*$, i.e.
$$P(\left(S^*_t\right)^{-1} \AA) = P(\AA)$$
for all Borel sets $\AA \subset \PP(\MM^*)$ and all $t >0$.
\item \emph{quasi-Palm} if a Borel set $\cA \subset \cM$ satisfies $P(\cA) = 1$ if and only if $P$ almost every $\nu$ satisfies $T_z^*\nu \in \cA$ for $\nu$ almost every $z$.
\item a \emph{fractal distribution} (FD) if it is scale-invariant and quasi-Palm.
\item an \emph{ergodic fractal distribution} (EFD) if it is a fractal distribution and it is ergodic under the action of $S_t^*$.
\end{enumerate}
Write $\FD$ and $\EFD$ for the set of all fractal distributions and ergodic fractal distributions, respectively.
\end{definition}

\begin{remark}
\begin{enumerate}
\item Hochman \cite{Hochman2010} used an alternative definition for the quasi-Palm property. The requirement was that any bounded open set $U$ containing the origin satisfies
$$
  P \sim \iint_U \delta_{T^*_x \mu} \, \textrm{d}\mu(x)\, \textrm{d}P(\mu).
$$
Both definitions are easily seen to agree, since two distributions are equivalent if and only if they have the same sets of full measure.

\item A distribution $P$ is called \textit{Palm} if for any open set $U$ containing the origin, we have
$$P = \iint_U \delta_{T_x \mu} \, \textrm{d}\mu(x)\, \textrm{d}P(\mu)$$
with finite intensity,
$$\int \mu(B_1) \, \mathrm{d}P(\mu) < \infty.$$
\end{enumerate}
\end{remark}

Note that in the above definitions $P$ is a distribution on $\PP(\MM)$, i.e.\ on measures with unbounded support. Most of the time we will need to deal with distributions supported on $\PP(B_1)$ instead (the main advantage being that this is a compact metrizable space).

\begin{notation}[Restricted distributions] Given $P$, we write $P^\square$ be the push-down of $P$ under $\mu\to\mu^\square$. Slightly abusing notation, whenever $P$ is an FD/EFD, we will also refer to $P^\square$ as an FD/EFD. Note that in this case $P^\square$ is $S^\square$-invariant, but the quasi-Palm is not properly defined for $P^\square$. When we want to emphasize whether we are talking about $P$ or $P^\square$, we will call the former the \emph{extended version} of $P$ and the latter the \emph{restricted version} of $P$.
\end{notation}

This abuse of notation  is justified by the following result; see \cite[Lemma 3.1]{Hochman2010}.

\begin{lemma}\label{lem:extended-restricted-correspondence}
The action $P\to P^\square$ induces a 1-1 correspondence between $S^*$-invariant and $S^\square$-invariant distributions.
\end{lemma}

We can now state the key result of Hochman \cite[Theorem 1.7]{Hochman2010} asserting the additional spatial invariance enjoyed by typical tangent distributions.

\begin{theorem} \label{thm:TDs-are-FDs}
For any $\mu\in\MM$ and $\mu$ almost every $x$, all tangent distributions of $\mu$ at $x$ are fractal distributions.
\end{theorem}

Given a fractal distribution $P$, as an invariant measure for dynamical system defined by the flow $S_t^*$ we can consider the \textit{ergodic decomposition} $\{P_\alpha\}$ of $P$ with respect to $S_t^*$:
$$P = \int P_\alpha \, \mathrm{d}P(\alpha).$$
Hochman \cite[Theorem 1.3]{Hochman2010} also proved that the quasi-Palm property is also preserved when passing to the ergodic components.

\begin{theorem} \label{thm:ergodic-components-of-FDs-are-FDs}
Almost all ergodic components of an FD are EFDs.
\end{theorem}

To conclude our discussion of fractal distributions, we note the following consequence of the Besicovitch density point theorem. A particular case of this is \cite[Proposition 3.7]{Hochman2010}; the proof is the same, so is omitted.

\begin{proposition} \label{prop:TD-restriction}
 If $\mu\in\cM$ and $\mu(A)>0$, then for $\mu$ almost all $x\in A$ we have that
 $$\TD(\mu,x)=\TD(\mu_A,x).$$
\end{proposition}

\subsection{CP distributions}
\label{subsec:CP}

CP processes, introduced by Furstenberg in \cite{Furstenberg2008} (though in embryonic form go back to \cite{Furstenberg1970}), are analogous to FDs, except that the zooming-in is done through $b$-adic cubes rather than cubes centered at the reference point.

\begin{notation}[Dyadic systems]
For simplicity we restrict ourselves to dyadic CP processes. Let $\DD$ be the partition of $B_1$ into $2^d$ cubes of the form $I_1\times\cdots\times I_d$, where $I_i\in\{ [-1,0),[0,1] \}$. Given $x\in B_1$, let $D(x)$ be the only element of $\DD$ containing it. More generally, for $k\ge 1$, we let $\DD_k$ be the collection of cubes of the form $I_1\times\cdots\times I_d$, where
\[
I_i\in \big\{ [-1,-1+2^{1-k}),\ldots, [1-2\cdot 2^{1-k}, 1 - 2^{1-k}), [1-2^{1-k},1] \big\}.
\]
We refer to elements of $\DD_k$ as dyadic cubes of step $k$ (or size $2^{1-k}$). Further, if $D$ is any cube, write $T_D$ for the orientation-preserving homothety mapping from $\overline{D}$ onto $B_1$.
\end{notation}

\begin{definition}[CP magnification operator] With this notation, we define the (dyadic) \textit{CP magnification operator} $M$ on $\PP(B_1)\times B_1$ by
\[
M(\mu,x)= M^\square(\mu,x) := (T_{D(x)}^\square \mu, T_{D(x)} (x)).
\]
This is defined whenever $\mu(D(x))>0$.
\end{definition}

Note that, unlike FDs, here it is important to keep track of the orbit of the point that is being zoomed upon. Note also that $M$ acts on $\Xi:=\PP(B_1)\times B_1$.

\begin{remark}
In Hochman's work, CP processes are defined via dyadic partitions of the half-open cube $[-1,1)^d$. This creates some technical issues with measures which give positive mass to the set $\{ (x_1,\ldots,x_n):x_i=1 \textrm{ for some } i\}$. Here we have followed Furstenberg's original definition from \cite{Furstenberg2008}. Because ultimately we will deal only with measures which give zero mass to the boundaries of cubes, this is just a matter of convenience.
\end{remark}

The analogue of the quasi-Palm in this context is the \textit{adaptedness} of distributions.

\begin{definition}[CP distributions]
A distribution $Q$ on $\Xi$ is \emph{adapted}, if there is a disintegration
\begin{equation} \label{eq:adapted}
\int f(\nu,x) \,\mathrm{d}Q(\nu,x) = \iint f(\nu,x)\,\textrm{d}\nu(x) \,\textrm{d}\overline{Q}(\nu)\quad\textrm{for all }f\in C(\Xi),
\end{equation}
where $\overline{Q}$ is the projection of $Q$ onto the measure component. Given a distribution $Q$ on $\Xi$, its \emph{intensity measure} is given by
\[
[Q](A) := \int \mu(A)\,\mathrm{d}\overline{Q}(\mu), \quad A \subset B_1.
\]
A distribution on $\Xi$ is a \emph{CP distribution} (CPD) if it is $M$-invariant (that is, $MQ=Q$), adapted, and its intensity measure is normalized Lebesgue measure on $B_1$, which we denote by $\mathcal{L}$. The family of all CP distributions is denoted by $\CPD$, and the ergodic ones by $\ECPD$.
\end{definition}

Note that adaptedness can be interpreted in the following way: in order to sample a pair $(\mu,x)$ from the distribution $Q$, we have to first sample a measure $\mu$ according to $\overline{Q}$, and then sample a point $x$ using the chosen distribution $\mu$. This interpretation highlights the connection with the quasi-Palm property.

\begin{remark}The condition that $[Q]$ is Lebesgue is not part of the definition of CP process given in \cite{Furstenberg2008,Hochman2010}, and indeed there are important examples of adapted, $M$-invariant distributions with non-Lebesgue intensity. However, all distributions we will consider do have this property, which will be required repeatedly in the proofs. As a first useful consequence, note that if $[Q]=\mathcal{L}$, then for any fixed cube $B(x,r)$, $Q$ almost all measures give zero mass to the boundary of $B(x,r)$. In particular, $Q$ almost every measure gives zero mass to the boundary of the elements of $\DD$. This will help us in dealing with the discontinuities inherent to the dyadic partition.

The usefulness of this condition was already implicit in Hochman's work \cite{Hochman2010}, where a random translation is often applied to ensure that the resulting CP processes have Lebesgue intensity. We also remark that, for us, CP processes are a tool towards the study of the scenery flow and fractal distributions, so we adopted the definition that happens to be most useful with this goal in mind.
\end{remark}

We can define concepts similar to the scenery and tangent distributions (Definition \ref{def:scenery-and-tangent-distributions}) for CP processes:

\begin{definition}[CP scenery and micromeasure distributions]
Given a measure $\mu \in \cM^\square$, $x \in B_1$, and $N \in \N$, we define the \textit{CP scenery distribution} of $\mu$ at $x$ along the scales $1,\dots,N$ by
$$\langle \mu,x \rangle_N := \frac1N \sum_{k = 0}^{N-1} \delta_{M^k(\mu,x)}.$$
Any accumulation point of $\langle \mu,x \rangle_N$ in $\cP(\Xi)$, as $N \to \infty$, is called a \textit{micromeasure distribution}, and the set of them is denoted by $\MD(\mu,x)$. We say that a measure $\mu \in \cM$ \textit{CP generates} $Q$ if
$$\langle \mu,x \rangle_N \to Q, \textrm{ as }N \to \infty,$$
that is, $\MD(\mu,x) = \{Q\}$, at $\mu$ almost every $x$.
\end{definition}

\begin{remark}
By compactness of $\PP(\Xi)$, the set $\MD(\mu,x)$ is always nonempty and compact.
\end{remark}

Again similarly as in Proposition \ref{prop:TD-restriction}, a consequence of the Besicovitch density point theorem (in its version for dyadic cubes, which is simpler and can be seen from a martingale argument) yields:

\begin{proposition} \label{prop:MD-restriction}
 If $\mu\in\cM^\square$ and $\mu(A)>0$, then for $\mu$ almost all $x\in A$ we have that
 $$\MD(\mu,x)=\MD(\mu_A,x).$$
\end{proposition}

Just like tangent distributions at typical points are fractal distributions (Theorem \ref{thm:TDs-are-FDs}), micromeasure distributions at typical points are CP distributions, but only after we randomly translate the measure. However, the only role of the random translation is to ensure that all micromeasure distributions have Lebesgue intensity.

\begin{theorem} \label{thm:micro-distris-are-CPDs}
Let $\mu\in\MM$. The following holds for $\mu$ almost all $x$.
\begin{enumerate}
\item All distributions in $\MD(\mu,x)$ are adapted.
\item If $Q\in\MD(\mu,x)$ has Lebesgue intensity, then $Q$ is a CPD.
\item For Lebesgue almost all $\omega\in B(0,1/2)$, all distributions in $\MD(\mu+\omega,x+\omega)$ are CPDs
\end{enumerate}
\end{theorem}
\begin{proof}
The first claim is \cite[Proposition 5.4]{Hochman2010}. The second follows from the proof of \cite[Proposition 5.5(2)]{Hochman2010}: although \cite[Proposition 5.5]{Hochman2010} is stated for random translations of a fixed measure, the second part only uses the fact that the intensity measure of $Q$ gives zero mass to all the boundaries of dyadic cubes. Finally, the last claim is precisely the content of \cite[Proposition 5.5]{Hochman2010}, except that there $\mu$ is assumed to be supported on $B(0,1/2)$, but after rescaling the measure we can extend the result to arbitrary $\mu\in\MM$.
\end{proof}

Just as the ergodic components of FDs are again FDs, ergodic components of CPDs are again CPDs:

\begin{proposition} \label{prop:ergodic-CPD}
Let $Q$ be a CPD.
\begin{enumerate}
\item For $\ol{Q}$ almost all $\mu$ and $\mu$ almost all $x$, we have that $\MD(\mu,x)=\{Q_{(\mu,x)}\}$, where $Q_{(\mu,x)}$ is the ergodic component of $(\mu,x)$
\item Almost all ergodic components of $Q$ are CPDs.
\end{enumerate}
\end{proposition}
\begin{proof}
Let $Q=\int Q_\alpha \,\mathrm{d}Q(\alpha)$ be the ergodic decomposition of $Q$. By the ergodic theorem, for a fixed $f\in C(\Xi)$ and $Q_\alpha$ almost all $(\mu,x)$,
\[
\lim_{n\to\infty} \int f\, \mathrm{d}\langle\mu,x\rangle_n = \int f\, \mathrm{d}Q_\alpha.
\]
Hence the same holds simultaneously for all $f$ in a uniformly dense countable subset of $C(\Xi)$, and therefore for all $f\in C(\Xi)$. This yields the first claim.

For the second claim, note that from the first part of Theorem \ref{thm:micro-distris-are-CPDs} and adaptedness of $Q$, it follows that for $Q$ almost all $(\mu,x)$, all the elements of $\MD(\mu,x)$ are adapted distributions. Thus, by the first part, $Q_\alpha$ is adapted for $Q$ almost all $\alpha$.

It remains to show that $[Q_\alpha]=\mathcal{L}$ for $Q$ almost all $\alpha$. For this, we will use some well-known facts on measure-theoretical entropy; \cite[Chapters 4 and 8]{Walters1982} contains all the definitions and facts we need.

Let $\Phi\colon B_1\to B_1$ equal to $T_D^{-1}$ on $D$ for each dyadic cube $D\in\DD$. Let $\Delta$ be the union of the boundaries of dyadic cubes of first level. Note that, since $Q$ is adapted, $[Q_\alpha](\Delta)=0$ for $Q$ almost all $\alpha$. Also,
% \begin{eqnarray*}
% \Phi[Q_\alpha](f) = \iint f(\Phi(x)) d\mu(x) \, dQ_\alpha(\mu,x)= \iint f(\Phi(x)) dQ_\alpha(\mu,x)= \iint f(x) dQ_\alpha(\mu,x)= [Q_\alpha](f).
% \end{eqnarray*}
\[
  \iint f(\Phi(x)) \,\textrm{d}\mu(x) \,\textrm{d}\overline{Q}_\alpha(\mu) = \iint f(x) \,\textrm{d}\mu(x) \,\textrm{d}\overline{Q}_\alpha(\mu)
\]
Thus each $[Q_\alpha]$ is $\Phi$-invariant. The map $\Phi$ on $B_1\setminus\Delta$ is naturally conjugated to the full shift on $2^d$ symbols. It follows that the system $(B_1, [Q_\alpha],\Phi)$ is measure-theoretically isomorphic to a measure-preserving system on the full shift on $2^d$ symbols, and the invariant measure on the latter is the measure of maximal entropy if and only if $[Q_\alpha]=\mathcal{L}$. In particular, denoting measure-theoretical entropy by $h_\nu(\Phi)$, we have $h_{[Q_\alpha]}(\Phi) \le d\log 2$, with equality if and only if $[Q_\alpha]=\mathcal{L}$. On the other hand, by the affinity of entropy,
\[
h_{[Q]}(\Phi) = \int h_{[Q_\alpha]}(\Phi)\, \mathrm{d}Q(\alpha).
\]
We conclude that $[Q_\alpha]=\mathcal{L}$ for $Q$ almost all $\alpha$, as claimed.
\end{proof}

\begin{lemma}\label{lem:CPconvex}
$\CPD$ is a convex subset of $\PP(\Xi)$, and the set of extremal points is exactly $\ECPD$.
\end{lemma}

\begin{proof}
The properties of adaptedness and having Lebesgue intensity are checked from definitions to be convex, so $\CPD$ is indeed a convex set. Ergodic CP distributions are extremal points of $\CPD$, since they are extremal points for the larger set of $M^\square$-invariant measures, and conversely, since we know from Proposition \ref{prop:ergodic-CPD} that the ergodic components of CPDs are CPDs.
\end{proof}

\subsection{Extended CP distributions} The operator $M$ has an extended version $M^*$, defined on $\cM\times B_1$ via
\[
M^*(\mu,x)=(T_{D(x)}^*\mu, T_{D(x)}x).
\]
We have the following analog of Lemma \ref{lem:extended-restricted-correspondence}.

\begin{lemma} \label{lem:restricted-to-extended-CP}
Given a CP distribution $Q$, there is an $M^*$-invariant distribution $\widehat{Q}$ on $\MM^*\times B_1$ such that $\widehat{Q}^\square = Q$, where $\widehat{Q}^\square$ is the push-down of $\widehat{Q}$ under $(\mu,x)\to (\mu^\square,x)$.
%Moreover, $Q$ is $M$-ergodic if and only if $\widehat{Q}$ is $M^*$-ergodic.
\end{lemma}
\begin{proof}
The lemma follows from \cite[Section 3.2]{Hochman2010}, but we give a complete proof as the construction will be used later. We can realize the system $(\Xi, M,Q)$ as a process $(\xi_n)_{n\in\N}$, where $\xi_1\sim Q$ and $M\xi_n = \xi_{n+1}$. By definition, this process is stationary with marginal $Q$ (note that given $\xi_n$, the future of the process $\{ \xi_m:m\ge n\}$ is deterministic). Any stationary one-sided process has an extension to a stationary two-sided process with the same finite-dimensional marginals; thus, let $(\xi_n)_{n\in \Z}$ be the two-sided extension of the above process, and denote its distribution by $\mathbb{P}_Q$. Then, in particular, $\mathbb{P}_Q$ almost surely it holds that $M\xi_n =\xi_{n+1}$ for all $n\in\Z$.

Suppose that a $\mathbb{P}_Q$-typical sequence $\xi_n = (\mu_n,x_n)$ is given. For each $n\ge 0$ and $x\in B_1$, let $T_{x,n}$ be the orientation-preserving homothety that maps the dyadic square of size $2\cdot 2^{-n}$ containing $x$ onto $B_1$ (this is well-defined for $x$ not in the boundary of a dyadic cube). Let $E_n= T_{x_{-n},n} B_1$. The sequence $(E_n)_{n\ge 0}$ is then an increasing sequence of compact sets, starting with $B_1$. Furthermore, $M^k(\mu_{-n},x_{-n})=(\mu_{-n+k},x_{-n+k})$  for all $n,k\in\N$. It follows from these considerations that the limit
\[
\nu = \lim_{n\to\infty} \nu_n = \lim_{n\to\infty} T_{x_{-n},n}^*\mu_{-n}
\]
exists, in the sense that the measures $\nu_n$ are supported on $E_n$ and are compatible: $\nu_{n+k}|_{E_n} =\nu_n$. Moreover, $\nu_0=\mu_0$.

Note that $(\nu,x_0)\in \MM^*\times B_1$ is a function of the sequence $(\mu_n,x_n)$; let $\widehat{Q}$ be the push down of $\mathbb{P}_Q$ under this map. One can check from the definitions that this is the desired extension: $\widehat{Q}$ is $M^*$-invariant, and $\widehat{Q}^\square=Q$.
\end{proof}

The distribution $\widehat{Q}$ is also called the extended version of $Q$. We note a consequence of the construction:
\begin{corollary} \label{cor:extended-CP-is-adapted}
Let $Q$ be a CP distribution, and let $\widehat{Q}$ be its extended version. If $f\colon B_1\to \R$ is a Borel function, then
\[
\int f(x) \,\mathrm{d}\widehat{Q}(\mu,x)= \int f(x) \,\mathrm{d}x.
\]
\end{corollary}
\begin{proof}
By the construction of the extended version,
\[
\int f(x) \,\mathrm{d}\widehat{Q}(\mu,x) = \mathbb{E}_{\mathbb{P}_Q} f(x_0) = \int f(x) \,\mathrm{d}Q(\mu,x) = \int f(x)\,\mathrm{d}x,
\]
using that the sequence $(\mu_n,x_n)$ has marginal $Q$ and that $Q$ is adapted and has Lebesgue intensity.
\end{proof}

\subsection{Weak convergence}
\label{subsec:weak_convergence}

To conclude this section we collect a number of standard facts on weak convergence. We will often have to prove weak convergence of distributions on $\Xi$. The following lemma shows that when considering convergence of CP sceneries, it is enough to establish convergence of the measure component.

\begin{lemma} \label{lem:measure-convergence-implies-CP-convergence}
Let $\overline{\langle \mu,x \rangle}_N$ be the projection of $\langle \mu,x \rangle_N$ onto the measure part, i.e.
\[
\overline{\langle \mu,x \rangle}_N := \frac1N \sum_{k = 0}^{N-1} \delta_{T_{D^k(x)}^\square \mu},
\]
where $D^k(x)$ is the dyadic cube of side length $2\cdot 2^{-k}$ containing $x$. If for some measure $\mu$ and a CPD $Q$ it holds that
\[
\overline{\langle \mu,x \rangle}_N \to \ol{Q} \quad\textrm{as } N\to\infty
\]
at $\mu$ almost every $x$, then $\mu$ CP generates $Q$.
\end{lemma}

\begin{proof}
By Theorem \ref{thm:micro-distris-are-CPDs}(1) and the hypothesis, for $\mu$ almost all $x$ any subsequential limit of  $\langle \mu,x \rangle_N$ is an adapted distribution with measure marginal $\overline{Q}$, hence it equals $Q$.
\end{proof}

The above lemma will be repeatedly used without further reference in the later sections. We will often need to use a metric which induces the weak topology on probability measures.

\begin{definition}\label{def:metriconmeasures} For any compact metric space $X$, we define a distance $d_X(\mu,\nu)$ between two finite measures $\mu$ and $\nu$ on $X$ by
\[
d_X(\mu,\nu) = \sup_{ f\in \Lip_1(X)} \int f\, \mathrm{d}(\mu-\nu),
\]
where $\Lip_1(X)$ is the class of Lipschitz functions $f : X \to \R$ with Lipschitz constant $1$ and $\| f \|_\infty \leq 1$.
\end{definition}

\begin{remark} \label{rem:wk-conv-bounded-f}
It is easy to see that $d_X$ is indeed a metric on the finite measures on $X$. It induces the weak topology (see e.g.\ \cite[Chapter 14]{Mattila1995}; the statement there is for $\R^d$ and a slightly different definition of the metric, but the proof extends to our situation with minor modifications). In fact, we will only need to know that the restriction of $d_X$ to $\cP(X)$ induces the weak topology on $\cP(X)$. Without the $\| f \|_\infty \leq 1$ condition, $d_X$ restricted to $\cP(X)$ is known as the $1$st Wasserstein metric, and it is easy to see that both metrics are equivalent (up to multiplicative constants) on $\cP(X)$, but we will have no use for this.

We will slightly abuse notation and denote by $d$ both the metric above on the space of measures $\PP(B_1)$ and on the space of distribution $\PP(\PP(B_1))$ as it is clear from the context which space we are dealing with.
\end{remark}

%Notice that we sometimes use an alternative definition of $d$ using $\delta$-neighborhoods of sets:
%\begin{eqnarray}\label{prokoequiv}
%d(\mu,\nu) = \inf\{\delta > 0 : \mu(A) \leq \nu(A_\delta) + \delta \textrm{ and } \nu(A) \leq \mu(A_\delta) + \delta \textrm{ for all Borel sets } A\},
%\end{eqnarray}
%where $A_\delta = \{y \in X : \dist(y,A) < \delta\}$, see \cite{Billingsley1968}. \red{Tuomas: this is the right reference if I recall correctly. Does anyone have this book so that I can give the exact theorem reference? My "online" book collection failed.}

Even though weak convergence is defined in terms of continuous functions, it still holds for functionals whose discontinuity set is null for the limiting measure.

\begin{lemma} \label{lem:weak-convergence}
Let $X$ be a locally compact metric space, and let $\mu_n, \mu\in\PP(X)$. If $\mu_n\to \mu$ weakly and $f\colon X\to \R$ is a function such that
\[
\mu(\{x\in X: f \textrm{ \emph{is discontinuous at} } x\}) = 0,
\]
then $\int f\,\mathrm{d}\mu_n\to \int f \,\mathrm{d}\mu$.
\end{lemma}

See e.g.\ \cite[Theorem 2.7]{Billingsley1968} for a stronger statement.

To finish this section, we show that to prove convergence of distributions in $\PP(B_1)$, it is enough to consider test functions ``with a finite resolution''. Recall that $\cD_k$ is the family of dyadic cubes of level $k$, and let $\cF_k$ be the class of functions $f\colon \PP(B_1)\to\R$ such that $f(\mu)$ depends only on the values of $\mu(D)$, $D\in\cD_k$.

\begin{lemma} \label{lem:easy-convergence}
Let $Q_n, Q\in\PP(B_1)$. If $\int f\,\mathrm{d}Q_n\to \int f\,\mathrm{d}Q$ for all $k\in\N$ and all $f\in  \cF_k$, then $Q_n\to Q$.
\end{lemma}

\begin{proof}
If the functions in $\cF_k$ were continuous, this would be a direct application of the Stone-Weierstrass Theorem. It would be possible to still rely on Stone-Weierstrass by approximating elements of $\cF_k$ by continuous functions in a suitable way, but we give a direct argument.

Given $\mu\in\PP(B_1)$, let $\mu_k=\sum_{D\in\cD_k} \mu(D)\delta_{z_D}$, where $z_D$ is the center of $D$. Note that if $\phi\colon B_1\to\R$ is $1$-Lipschitz, then $|\phi(z)-\phi(z_D)|\le \sqrt{d} 2^{-k}$ for any $z\in D$, and therefore
\begin{eqnarray*}
\left|\int \phi \,\mathrm{d}\mu-\int \phi \,\mathrm{d}\mu_k \right|&=& \left|\sum_{D\in\cD_k} \mu(D)\left(\frac{\int_D \phi\,\mathrm{d}\mu}{\mu(D)}-  \phi(z_D)\right) \right|\\
&\le& \sum_{D\in\cD_k} \mu(D) \sqrt{d} 2^{-k} = \sqrt{d} 2^{-k}.
\end{eqnarray*}
This shows that $d(\mu,\mu_k)\le \sqrt{d} 2^{-k}$.

Now let $f\in C(\PP(B_1))$, and write $f_k(\mu)=f(\mu_k)$. Then $f_k\in\cF_k$ by definition, and hence, by the hypothesis,
\[
\lim_{n\to\infty} \int f_k\, \mathrm{d}Q_n = \int f_k\, \mathrm{d}Q.
\]
Since $f$ is continuous, it is uniformly continuous; hence given $\e>0$ there is $k>0$ such that $|f(\mu)-f(\nu)|<\e$ if $d(\mu,\nu)< \sqrt{d} \, 2^{-k}$. In particular, by the above, $|f(\mu)-f_k(\mu)|<\e$ for all $\mu\in\PP(B_1)$. By writing $f=(f-f_k)+f_k$, it follows from the above that
\[
\limsup_{n\to\infty} \int f\, \mathrm{d}Q_n \le \int f\,\mathrm{d}Q+\e,
\]
and likewise with $\liminf$. Since $\e>0$ was arbitrary, this completes the proof.
\end{proof}

\section{Fractal distributions form a closed set}
\label{sec:FDs-closed}

In this section, we prove the following precise version of Theorem \ref{thm:FDs-are-closed}.

\begin{theorem} \label{thm:FDs-are-closed-2}
 The set $\FD$ is closed in the weak topology.
\end{theorem}

The proof relies on the results in \cite{Hochman2010} relating FDs to CP distributions, which we first recall. The main tool is the \textit{centering} operation that will provide a way to map CPDs onto FDs, and vice versa.

\begin{definition}[Centering operation]
Let $C\colon \MM^*\times B_1\times [0,\log 2]\to \MM^*$ be given by
$$C(\mu,x,t)=S_t^* T_x\mu.$$
If $Q$ is a distribution on $\MM^*\times B_1$, then its (continuous) \emph{centering} $\cent(Q)$ is the push-down of $Q\times \lambda$ under $C$, that is,
$$\cent(Q) := C(Q \times \lambda),$$
where $\lambda$ is normalized Lebesgue measure on $[0,\log 2]$.
\end{definition}
\providecommand{\cent}{\mathop{\rm cent}\nolimits}
\begin{theorem} \label{thm:equivalence-CP-FD}
Let $Q$ be a CP distribution, and let $\widehat{Q}$ be its extended version (given by Lemma \ref{lem:restricted-to-extended-CP}). Then $\cent(\widehat{Q})$ is an extended FD.

Conversely, given an extended fractal distribution $P$, there exists  a CP distribution $Q$ such that $P=\cent(\widehat{Q})$, with $\widehat{Q}$ the extended version of $Q$.
\end{theorem}

\begin{proof}
See \cite[Theorem 1.14, Theorem 1.15, and Proposition 1.16]{Hochman2010}. We remark that although in \cite{Hochman2010} CPDs are not required to have Lebesgue intensity measure, Proposition 1.16 states that one can find an appropriate CPD with this additional property.
\end{proof}

This correspondence theorem between CPDs and FDs allows us to reduce the investigation back to CP distributions. In this ``discrete'' setting, the analogue of Theorem \ref{thm:FDs-are-closed-2} is quite straightforward:

\begin{lemma} \label{lem:CP-distributions-are-closed}
$\CPD$ is closed in $\PP(\Xi)$.
\end{lemma}

\begin{proof}
Write $\mathcal{U}$ for the distributions in $\PP(\Xi)$ with Lebesgue intensity measure. Note that $Q\in\mathcal{U}$ if and only if
\[
\iint f \,\textrm{d}\nu \,\textrm{d}\overline{Q}(\nu) = \int f(x) \,\mathrm{d}x\quad \textrm{for all }f\in C(B_1).
\]
The left-hand side defines a continuous function of $Q$, so $\mathcal{U}$ is a closed set. For fixed $f\in C(\Xi)$, both sides of (\ref{eq:adapted}) are continuous as a function of $Q$, so the family of adapted distributions is also closed.

It remains to show that if $Q_n$ are CPDs and $Q_n\to Q$, then $Q$ is $M$-invariant. Note that $M$ is discontinuous in general, however it is discontinuous only at pairs $(\mu,x)$ where $x$ has some coordinate equal to $0$ (i.e.\ $x$ is in the boundary of two dyadic cubes of first level). Since we already know that $Q\in\mathcal{U}$ and $Q$ is adapted, $M$ is continuous off a set of $Q$-measure zero. Lemma \ref{lem:weak-convergence} then tells us that $MQ=\lim_{n\to\infty} M Q_n=Q$, as desired.
\end{proof}

\begin{figure}[t!]
\label{fig:centering}
\includegraphics[scale=0.65]{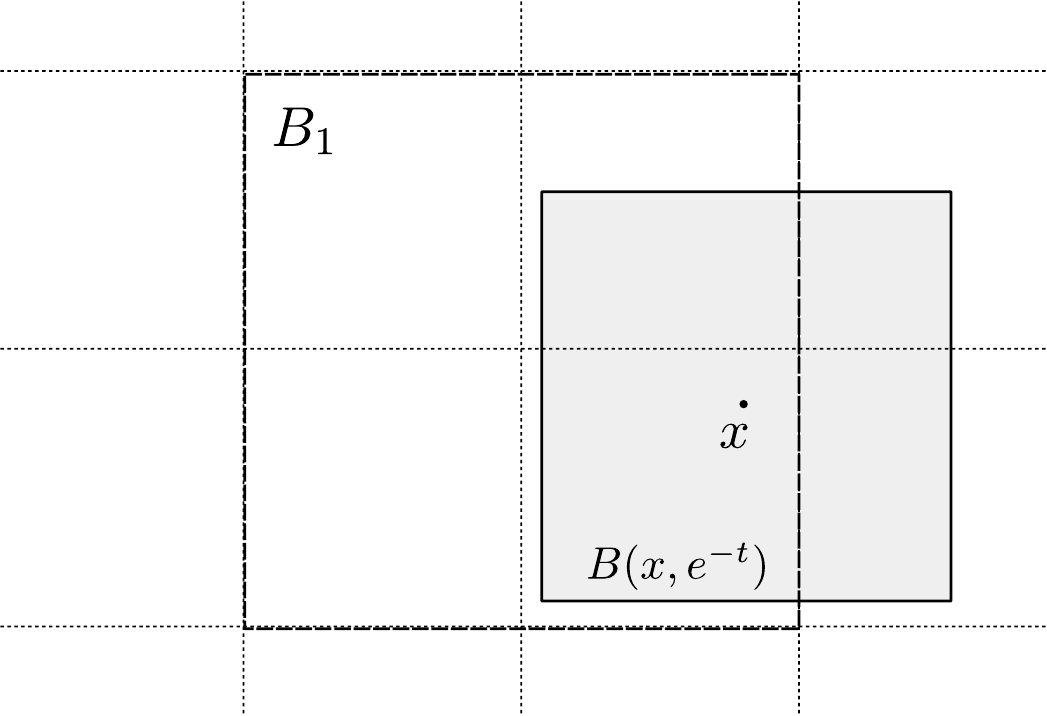}
\caption{The centering $\cent(Q)$ provides a distribution of magnifications $C(\mu,x,t)$ where the scale $e^{-t}$ is chosen uniformly between $1/2$ and $1$ and $(\mu,x)$ according to the distribution $Q$.}
\end{figure}

\begin{proof}[Proof of Theorem \ref{thm:FDs-are-closed-2}]
Let $\{ P_n\}$ be a sequence of restricted FDs such that $P_n\to P$ in  $\PP(B_1)$. For each $n$, the second part of Theorem \ref{thm:equivalence-CP-FD} provides us with a CP distribution $Q_n$  such that $\cent(\widehat{Q}_n)=\widehat{P}_n$, where $\widehat{P}_n, \widehat{Q}_n$ are the corresponding extended versions. By Lemma \ref{lem:CP-distributions-are-closed} and compactness, there exists a CP distribution $Q$ which is an accumulation point of the $Q_n$. Let $\widehat{Q}$ be the extended version of $Q$. It is enough to prove that $\cent(\widehat{Q})^\square=P$, since then the first part of Theorem \ref{thm:equivalence-CP-FD} will yield that $P$ is a restricted FD. In turn, this will follow if we can prove that the map $Q\to\cent(\widehat{Q})^\square$ is continuous on CP distributions.

Note that $C(\mu,x,t)$ depends only on the restriction of $\mu$ to $B_2=B(0,2)=[-2,2]^d$. Let
$$\mu^\Diamond := \mu_{B_2},$$
with $Q^\Diamond, \MM^\Diamond$, etc being defined in the usual way. The desired continuity will then follow if we can establish the following two claims:
\begin{enumerate}
\item The map $Q\to \widehat{Q}^\Diamond$ is continuous from the set of CP distributions to $\cP(\MM^\Diamond\times B_1)$.
\item The map $Q\to \cent(Q)^\square$ is continuous from $\cP(\MM^\Diamond\times B_1)$ to $\cP(\MM)$.
\end{enumerate}
These claims are proved in the following two lemmas.
\end{proof}

\begin{lemma} \label{lem:restricted-to-extended-is-continuous}
The map $Q\to \widehat{Q}^\Diamond$ is continuous on the set of CP distributions.
\end{lemma}

\begin{proof}
Let $f\in C(\MM^\Diamond\times B_1)$. We have to show that if $Q_k, Q$ are CP distributions, and $Q_k\to Q$, then
\[
\int f(\nu^\Diamond,x) \,\mathrm{d}\widehat{Q}_k(\nu,x) \to \int f(\nu^\Diamond,x) \,\mathrm{d}\widehat{Q}(\nu,x) .
\]
We may and will assume that $|f|$ is uniformly bounded by $1$. Fix $n\ge 2$, and write $\Delta_n = \{ y: \dist(y,\partial B_1)\ge 2\cdot 2^{-n}\}$. Decompose
\[
\int f(\nu^\Diamond,x) \,\mathrm{d}\widehat{Q}(\nu,x) = \int_{x\in B_1\setminus\Delta_n} f(\nu^\Diamond,x) \, \mathrm{d}\widehat{Q}(\nu,x) +  \int_{x\in \Delta_n} f(\nu^\Diamond,x) \,\mathrm{d}\widehat{Q}(\nu,x).
\]
Using that $\|f\|_\infty\le 1$ and Corollary \ref{cor:extended-CP-is-adapted}, it follows that
\[
\left| \int_{x\in B_1\setminus \Delta_n} f(\nu^\Diamond,x)\, \mathrm{d}\widehat{Q}(\nu,x)\right| \le \int \mathbf{1}_{B_1\setminus\Delta_n}(x)\, \mathrm{d}\widehat{Q}(\nu,x)
\le C_d\, 2^{-n},
\]
and likewise for $\widehat{Q}_k$, where $C_d>0$ depends on the dimension $d$ only. Thus it is enough to show that, for a fixed $n$,
\[
\lim_{k\to\infty} \int_{x\in \Delta_n} f(\nu^\Diamond,x)\, \mathrm{d}\widehat{Q}_k(\nu,x) = \int_{x\in \Delta_n} f(\nu^\Diamond,x) \,\mathrm{d}\widehat{Q}(\nu,x).
\]

We use the notation of the proof of Lemma \ref{lem:restricted-to-extended-CP}. Let $(\mu_n,x_n)$ be a $\mathbb{P}_Q$-typical point, and let $(\nu,x_0)$ be the resulting element of $\MM\times B_1$. Note that if $x_{-n}\in\Delta_n$, then $T_{x_{-n},n}B_1 \supset B_2$, and it follows that $\nu^\Diamond= T_{x_{-n},n}^\Diamond\mu_{-n}$. Hence, keeping in mind the construction of the extended version given in Lemma \ref{lem:restricted-to-extended-CP},

\begin{eqnarray*}
\int_{x\in \Delta_n} f(\nu^\Diamond,x) \,\mathrm{d}\widehat{Q}(\nu,x) &=& \mathbb{E}_{\mathbb{P}_Q} \mathbf{1}_{\Delta_n}(x_{-n}) f(T_{x_{-n},n}^\Diamond\mu_{-n}, T_{x_{-n},n} x_{-n})\\
&=& \int \mathbf{1}_{\Delta_n}(x)f(T_{x,n}^\Diamond\mu, T_{x,n} x)\, \mathrm{d}Q(\mu,x),
\end{eqnarray*}
where in the last step we used that the stationary sequence $(\mu_n,x_n)$ has marginal $Q$. Likewise, the same holds for $Q_k$ in place of $Q$. The function $(\mu,x)\to \mathbf{1}_{\Delta_n}(x) f(T_{x,n}^\Diamond\mu, T_{x,n} x)$ is continuous except for some pairs $(\mu,x)$ with $x$ at the boundary of two dyadic cubes of side length $2\times 2^{-n}$. Since $Q$ is adapted and has Lebesgue intensity, it gives zero mass to this discontinuity set. We are done thanks to Lemma \ref{lem:weak-convergence}.
\end{proof}

\begin{lemma} \label{lem:centering-is-continuous}
The map $Q\to \cent(Q)^\square$ is continuous from $\cP(\MM^\Diamond\times B_1)$ to $\cP(B_1)$.
\end{lemma}

\begin{proof}
Let $f\in C(\MM^\square)$ and $Q\in\PP(\MM^\Diamond\times B_1)$. By the definition of centering and Fubini,
\[
\int f\, \mathrm{d}\left(\cent(Q)^\square\right) =  \iint_0^{\log 2} f(S_t^\square T_x^*\mu) \,\mathrm{d}\lambda(t) \, \mathrm{d}Q(\mu,x).
\]
Even though each $S_t^\square$ may be discontinuous (when the boundary of $B(0,e^{-t})$ has positive mass), for a given measure there can be discontinuities only for countably many values of $t$. It follows from the bounded convergence theorem that the inner integral is a continuous function of $T_x^*\mu$, which in turn is a continuous function of $(\mu,x)$, and the lemma follows.
\end{proof}

\section{Splicing and the simplex of fractal distributions}
\label{sec:Poulsen}

The goal of this section is to establish Theorem \ref{thm:Poulsen}, which is precisely stated as follows.

\begin{theorem} \label{thm:Poulsen-2}
 The convex set $\FD$ is a Poulsen simplex (as a subset of the locally convex space of finite Radon measures). In other words, extremal points of $\FD$ are weakly dense in $\FD$.
\end{theorem}

Again, invoking the centering operation all we need to prove is the following:
\begin{proposition} \label{prop:CPD-Poulsen}
The set $\CPD$ is a Poulsen simplex.
\end{proposition}

We show how to deduce Theorem \ref{thm:Poulsen-2} from this proposition; the remainder of this section is devoted to the proof of the proposition.

\begin{proof}[Proof of Theorem  \ref{thm:Poulsen-2} (Assuming Proposition \ref{prop:CPD-Poulsen})]
By Theorem \ref{thm:equivalence-CP-FD}, the centering operation maps CP distributions onto fractal distributions. The centering map is in fact a factor map from the suspension flow with base $(\MM\times B_1,M)$ and height $\log 2$ onto $\FD$; see \cite{Hochman2010} for the details. It follows that if $Q$ is an ergodic extended CPD, then its centering is an ergodic FD. Moreover, if a CPD is ergodic then so is its extended version. Indeed, the construction in Lemma \ref{lem:restricted-to-extended-CP} shows that the extended version is a factor of the two-sided extension of $Q$, which is ergodic if and only if $Q$ is ergodic.

Finally, we recall from Lemmas \ref{lem:restricted-to-extended-is-continuous} and \ref{lem:centering-is-continuous} that the map $Q\to \cent(\widehat{Q})^\square$ is continuous. Thus the image of the dense set of ergodic CPDs under this map is dense in the set of restricted FDs and, by the above observations, consists of ergodic distributions, concluding the proof.
\end{proof}

\subsection{The splicing operation}
\label{subsec:splicing}
The construction of the ergodic CPD which approximates a given CPD will be done via an operation which we term the \emph{splicing} of scales. To introduce the notation, it will be convenient to identify points in $B_1$ with dyadic sequences:

\begin{notation}[Coding dyadic cubes]\label{coding} Write $\cA = \{0,1,\dots,2^d - 1\}$. Enumerate $\mathcal{D}$ (the dyadic sub-cubes of $B=B_1$ of first level, recall Section \ref{subsec:CP}) as $\{ B_i : i \in \cA\}$. Each $x \in \cA^k$ then corresponds to a dyadic cube $B_{x}$ of generation $k$ and side-length $2^{1-k}$. We will silently identify $x$ with $B_{x}$ whenever there is no possibility of confusion. Moreover, we will also identify each $x \in B$ with $x \in \cA^\N$ such that
\[
\{x \}= \bigcap_{k \in \N}B_{x|_k}.
\]
If $x$ is any sequence of length $\ge b$ (possibly infinite), we will write $x_a^b = (x_{a+1}, x_{a+2},\ldots, x_b)$. Also, if $x$ if an infinite word, we write $x_n^\infty = (x_{n+1},x_{n+2},\ldots)$; geometrically, $x_n^\infty = T_{B_{x_0^n}}x$. In the case $a=0$, we also write $x|_b=x_0^b$. We allow the empty word $\varnothing$, which in our identification corresponds to $B_1$;  we note $x_a^a=\varnothing$. If $\mu\in\PP(B)$ and $x\in\cA^k$, we will write
$$\mu(\cdot|x)=T_{B_x}\mu_{B_x}.$$
In symbolic notation, if $y\in\cA^\ell$, then $\mu(y|x)= \mu(xy)/\mu(x)$.
\end{notation}

We can now give the definition of the splicing map.

\begin{definition}[Splicing map] Given a sequence $\mathbf{n}=(n_i)_{i\in\N}$ of integers, we define $\spl_{\mathbf{n}}\colon B_1^\N\to B_1$ as
\[
 \spl_{\mathbf{n}}((x^i)_{i\in \N}) := (x^1|_{n_1}x^2|_{n_2}\cdots),
\]
where the notation on the right-hand side indicates concatenation of words.
\end{definition}

\begin{figure}[t!]
\label{fig:splicing}
\includegraphics[scale=0.8]{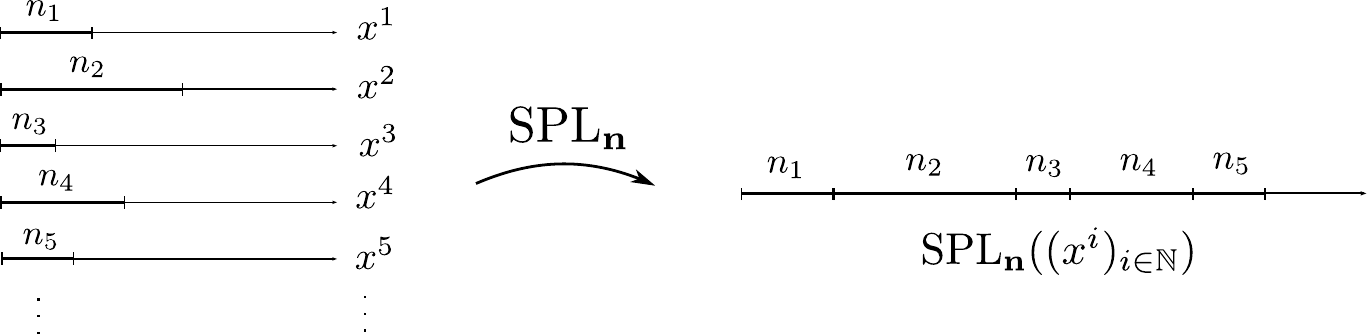}
\caption{Given $\n = (n_i)_{i \in \N}$, the splicing map $\spl_\n$ takes the first $n_i$ coordinates of the word $x^i \in B_1$ and concatenates the obtained finite words into a point in $B_1$.}
\end{figure}

Of particular interest to us will be the action of the splicing map on product measures $\times_{i=1}^\infty \mu^i$ on $B_1^\N$, and also on product distributions. Let
\[
\nu=\spl_{\mathbf{n}}(\times_{i=1}^\infty \mu^i).
\]
By the definition of splicing, the $\nu$-mass of a finite word $y$ (or equivalently, a dyadic cube $D$) is built from the $\mu_i$-masses of consecutive sub-words of $y$ whose length comes from the sequence $\n$. To make this statement precise, for $k \in \N$, denote the sum
\[
S_k = S_k(\n) := n_1+n_2+\cdots+n_k.
\]
\begin{lemma}
\label{lem:simplersplicing}
If $k \in \N$, $S_k\le N \le S_{k+1}$ and $y\in\cA^N$, then
\begin{equation} \label{eq:property-splicing}
 \nu(y) = \mu^1(y_0^{S_1})\mu^2(y_{S_1}^{S_2})\cdots \mu^{k-1}(y_{S_{k-1}}^{S_k}) \mu^k(y_{S_k}^N).
\end{equation}
\end{lemma}
\begin{proof}
If $x\in\cA^{S_k+i}$ and $y\in\cA^j$, with $0\le i\le i+j\le n_{k+1}$, then
\begin{equation} \label{eq:local-splicing}
\nu(y|x) = \mu^k(y|x_{S_k}^{S_k+i}).
\end{equation}
In particular, if $x\in\cA^{S_k}$ and $y\in\cA^j$ with $0\le j\le n_{k+1}$, then
$\nu(y|x) = \mu^k(y)$.
By iterating (\ref{eq:local-splicing}), if $S_k\le N \le S_{k+1}$ and $y\in\cA^N$, then
$$\nu(y) = \mu^1(y_0^{S_1})\mu^2(y_{S_1}^{S_2})\cdots \mu^{k-1}(y_{S_{k-1}}^{S_k}) \mu^k(y_{S_k}^N)$$
as claimed.
\end{proof}

Relying on this lemma, by choosing a suitable sequence $\n$, we can now control the frequency of occurrences of the measures $\mu^i$ in the CP scenery $\la \nu,x \ra_{N}$.

\subsection{Proof of Proposition \ref{prop:CPD-Poulsen}}

We are now ready to establish Proposition \ref{prop:CPD-Poulsen}. We know from Lemma \ref{lem:CP-distributions-are-closed} that $\CPD$ is compact. Moreover, by Lemma \ref{lem:CPconvex} and the existence and uniqueness of the ergodic decomposition, $\CPD$ is a Choquet simplex. Thus for Proposition \ref{prop:CPD-Poulsen} we only need to show that $\ECPD$ is dense in $\CPD$:

\begin{proposition}\label{prop:denseECPD}
Ergodic CPDs are dense in $\CPD$.
\end{proposition}

\begin{proof}
The density of ergodic CPDs is implied by the Krein-Milman Theorem if we are able to prove that, given a rational probability vector $(t_1/q,\ldots, t_k/q)$, and given ergodic CPDs $R_1,\ldots,R_k$, there is a sequence of ergodic CPDs $Q^N$ converging to $\frac1q \sum_{i=1}^k t_i R_i$ as $N \to \infty$.

To find the sequence $Q^N$, let $\spl=\spl_{\mathbf{n}}$ be the splicing map corresponding to the following $k$-periodic sequence
\[
 \mathbf{n} = (N t_1, Nt_2,\ldots, N t_k )^\infty = (N t_1, \ldots, N t_k, N t_1,\ldots, N t_k,\ldots).
\]
Note that, as $t_i/q$ is rational, $Nt_i$ is an integer for all $i \in \{ 1,\dots,k \}$. Write $\widetilde{R}=\times_{i=1}^k R_i$ and define an adapted distribution $P=P^N$ by setting
$$\ol{P}=\spl(\widetilde R^\N),$$
that is, we take the distributions $R_i$ in the product $k$-periodically. Then define $Q^N$ by setting the measure marginal
$$\ol{Q}^N=\frac1N \sum_{j=0}^{N-1} \ol{Q}_j,$$
where $\ol{Q}_j$ is the push-forward of $(\Xi,P)$ under the map $(\mu,x)\to \mu(\cdot|x_0^j)$. Then $Q^N$ is an ergodic CPD with $Q^N \to \frac{1}{q}\sum_{i=1}^k t_i R_i$ as $N \to \infty$. These two facts are verified in the Lemma \ref{lem:splicing-is-ECPD} and Lemma \ref{lem:splicings-converge} below.
\end{proof}

\begin{lemma} \label{lem:splicing-is-ECPD}
For fixed $N$, the distribution $Q=Q^N$ is an ergodic CPD.
\end{lemma}
\begin{proof}
The outline of the proof is simple: we show that $Q$ has Lebesgue intensity by definition, and then that $\MD(\mu,x)=\{Q\}$ for $Q$ almost all $\mu$ and $\mu$ almost all $x$. Then it follows from Theorem \ref{thm:micro-distris-are-CPDs} that $Q$ is a CPD (here we need to know that $Q$ has Lebesgue intensity), and then from Proposition \ref{prop:ergodic-CPD} that $Q$ is ergodic. We proceed to the details.

We start by showing that $P$ has Lebesgue intensity. Since each $R_i$ has Lebesgue intensity, using (\ref{eq:property-splicing}) and the $k$-periodicity of the sequence $\n$, we find that if $K=S_{\ell k}$ for some $\ell\in\N$, then
\begin{eqnarray*}
[P](x_0^K) = \int \nu(x_0^K)\,\mathrm{d}\ol{P}(\nu) &=& \prod_{j=1}^{\ell} \prod_{i = 1}^k \int \mu(x_{S_{j+i-1}}^{S_{j+i}}) \,\mathrm{d} \ol{Q}_{j+i}(\mu) \\
&=& \prod_{j=1}^{\ell} \prod_{i = 1}^k \int \mu(x_{S_{i-1}}^{S_i}) \,\mathrm{d} \ol{R}_i(\mu) \\
&=& \Big( \prod_{i = 1}^k 2^{-d N t_i}\Big)^\ell =  2^{-dK}.
\end{eqnarray*}
Since cubes of the form $x_0^{K}$ generate the Borel $\sigma$-algebra, $P$ has Lebesgue intensity as claimed. Now
\begin{eqnarray*}
[Q_j](x_0^\ell) =  \int \mu(x_0^\ell) \, \mathrm{d} \ol{Q}_j(\mu)&=& \int \mu(x_0^\ell|y_0^j) \,\mathrm{d}\mu(y) \, \mathrm{d}\ol{P}(\mu)\\
&=&  \int\frac{\mu((y_0^j)(x_0^{\ell}))}{\mu(y_0^j)} \, \mathrm{d}\mu(y)\, \mathrm{d}\ol{P}(\mu)\\
&=&  \sum_{z\in\mathcal{A}^j} \int \mu((zx)_0^{j+\ell}) \, \mathrm{d}\ol{P}(\mu)= \mathcal{L}(x_0^\ell),
\end{eqnarray*}
showing that $Q_j$ and therefore the average $Q$ also has Lebesgue intensity.

Next, we claim that
\begin{equation} \label{eq:P-empirical-measure}
\lim_{L\to\infty} \frac1L \sum_{i=0}^{L-1} \delta_{\mu(\cdot|x_0^{i N})}  = \overline{P}\quad\textrm{for }P\textrm{ almost all }(\mu,x).
\end{equation}
In essence this is a consequence of the ergodic theorem for product measures under the shift. After re-indexing, the splicing map $\spl \colon B_1^\N \to \N$ induces a map $\spl \colon (B_1^k)^\N \to B_1$ on the space of $k$-tuples $(B_1^k)^\N$ by
\[
\spl(\eta) = (\Phi(\eta_1)\Phi(\eta_2)\cdots), \quad \eta = (\eta_1,\eta_2,\dots) \in (B_1^k)^\N,
\]
where for a given $\zeta = (x^1,\ldots, x^k)\in B_1^k$, we define
\[
\Phi(\zeta)=(x^1|_{N t_1}\cdots x^k|_{N t_k}).
\]
Using this description, we can reformulate (\ref{eq:P-empirical-measure}) as
\[
\lim_{L\to\infty} \frac1L \sum_{i=0}^{L-1} \delta_{\spl(\sigma^i \eta)} = \spl\big(\widetilde{R}^\N\big)\quad\textrm{for }\widetilde{R}^\N\textrm{ almost all }\eta,
\]
where $\sigma$ is the shift on the sequence space $(B_1^k)^\N$. But this is a consequence of the ergodic theorem applied to the ergodic system $\big((B_1^k)^\N, \widetilde{R}^\N\big)$. Indeed, we need to show that for any $f\in C(B_1)$ and $\widetilde{R}^\N\textrm{ almost all }\eta$,
\[
\lim_{L\to\infty} \frac1L \sum_{i=0}^{L-1} f(\spl(\sigma^i \eta)) = \int f\,\mathrm{d}\spl\big(\widetilde{R}^\N\big).
\]
In turn, it is enough to verify this for $f$ in a countable dense subset of $C(B_1)$, and hence for a fixed $f\in C(B_1)$. But this holds by the ergodic theorem applied to the function $f\circ \spl$.

Next, we claim that
\begin{equation} \label{eq:convergence-to-Qj}
\lim_{L\to\infty} \frac1L \sum_{i=0}^{L-1} \delta_{\mu(\cdot|x_0^{i N+j})}  = \ol{Q}_j\quad\textrm{for }P\textrm{ almost all }(\mu,x).
\end{equation}
We start by noting that
\[
\mu(\cdot|x_0^{iN+j}) = \mu(\cdot|x_0^{iN})(\cdot|x_{iN}^{iN+j}).
\]
(The notation on the right-hand side means $\nu(\cdot|x_{iN}^{iN+j})$ where $\nu=\mu(\cdot|x_0^{iN})$.) Indeed, it is straightforward to check the equality for cubes $[z]$ which form a basis of the $\sigma$-algebra.

Given $\eta\in (B_1^k)^\N$, let
\[
 \Psi(\eta) = \spl(\eta)(\cdot|\Phi(\eta_1)_0^j).
\]
Using our previous notation, the last observation, and the definitions of $\overline{P}$ and $\overline{Q}_j$, we find that (\ref{eq:convergence-to-Qj}) is equivalent to
\[
\lim_{L\to\infty} \frac1L \sum_{i=0}^{L-1} \delta_{\Psi(\sigma^i \eta)} = \Psi\big(\widetilde{R}^\N\big)\quad\textrm{for }\widetilde{R}^\N\textrm{ almost all }\eta.
\]
Just as before, this follows from the ergodic theorem. Averaging over $j$, we conclude that
\[
\lim_{L\to\infty} \frac1L \sum_{i=0}^{L-1} \delta_{\mu(\cdot|x_0^i)}  = \ol{Q}\quad\textrm{for }P\textrm{ almost all }(\mu,x).
\]
Now from Theorem \ref{thm:micro-distris-are-CPDs}(2) we deduce that $Q$ is a CPD. Since a full $P$-measure set has positive $Q$-measure, the second part of Proposition \ref{prop:ergodic-CPD} shows that $Q$ is ergodic, finishing the proof.
\end{proof}

\begin{lemma} \label{lem:splicings-converge} It holds that
\[
\lim_{N\to\infty} Q^N = \frac{1}{q}\sum_{i=1}^k t_i R_i .
\]
\end{lemma}
\begin{proof}
By Lemma \ref{lem:easy-convergence}, we only have to prove that if $f\in\cF_p$ for some $p$, then
\[
\lim_{N\to\infty} \int f\, \mathrm{d}Q^N = \frac1q \sum_{i=1}^k t_i \int f\, \mathrm{d}R_i.
\]
Recall that $\cF_p$ is the class of function $f \colon \cP(B_1) \to \R$ such that $f(\mu)$ depends only on the values $\mu(D)$, $D \in \cD_p$. Define the set of indices
$$\mathcal{G}_{N,i} = \{j\in \{0,\ldots, N-1\} : N(t_1+\ldots+t_{i-1})\le j \le j+p\le N(t_1+\ldots+t_i)\}.$$
Note that (since we are keeping $p$ fixed),
\[
\lim_{N\to\infty} \frac{\#\mathcal{G}_{N,i}}{t_i N}=1 \quad \textrm{for all }i \in \{ 1,\ldots,k \}.
\]
On the other hand, if $j\in\mathcal{G}_{N,i}$, then it follows from (\ref{eq:local-splicing}) and the definition of $\cF_p$ that
\[
\int f\, \mathrm{d} Q_j^N = \int f(\mu(\cdot|y_0^j))\, \mathrm{d}\mu(y)\, \mathrm{d}P(\mu) = \int f(\mu(\cdot|y_0^j)) \,\mathrm{d}\mu(y)\, \mathrm{d}R_j(\mu) = \int f \, \mathrm{d}R_j,
\]
using that $R_j$ is a CPD in the last equality. We conclude that
\[
\left| \int f\, \mathrm{d}Q^N -  \frac1q \sum_{i=1}^k t_i \int f\, \mathrm{d}R_i\right| \le \|f\|_\infty \frac{\sum_{i=1}^k |t_i N-\#\mathcal{G}_{N,i}|}{N} \to 0
\]
as $N\to\infty$.
\end{proof}

\section{Every FD is generated by a USM}
\label{sec:USM}

In this section we establish Theorem \ref{thm:measure-generating-FD}, which we restate as follows:

\begin{theorem} \label{thm:measure-generating-FD-2}
 For any $P\in \FD$ there is a uniformly scaling measure $\mu$ which generates $P$. In other words, there is a Radon measure $\mu$ such that $\TD(\mu,x)=\{P\}$ for $\mu$ almost all $x$.
\end{theorem}

Firstly we notice that for ergodic fractal distributions this is a consequence of the ergodic theorem.

\begin{lemma}
Let $P$ be an FD. Then $P$ almost all $\mu$ are USM generating the ergodic component of $P_\mu$. In particular, if $P$ is an EFD, then $P$ almost all measures generate $P$.
\end{lemma}
\begin{proof}
This follows form \cite[Theorem 3.9]{Hochman2010} and the ergodic decomposition.
\end{proof}

In particular, if $P$ is an EFD, then there exists at least one measure generating $P$. If $P$ is not ergodic, this is still true, but requires a more involved argument using the splicing operation introduced in the previous section.

Yet again, the corresponding statement for CPDs is easier to prove, and implies Theorem \ref{thm:measure-generating-FD-2} by invoking the centering operation.

\begin{proposition} \label{CPtoFD}
If $\mu \in \MM$ and $Q$ is a CP distribution, then at $\mu$ almost every $x$ where $Q \in \MD(\mu,x)$, also $P = \cent(\widehat{Q})^\square \in \TD(\mu,x)$, where $\widehat{Q}$ is the extended version of $Q$. In fact, if $\mu$ CP generates $Q$, then $\mu$ is a USM generating $P$.
\end{proposition}

\begin{proof}
This is essentially proved in the course of the proof of \cite[Proposition 5.5(3)]{Hochman2010}. Although in that proposition the setting is that of an arbitrary measure that has been translated by a random vector, in the proof of the third part what really gets proved is that if $\mu$ is a measure such that for $\mu$ typical $x$, the sequence $\la \mu,x\ra_{N_i}$ converges to a CP distribution $Q$ along some sequence $(N_i)$, and $\langle \mu \rangle_{x,T_i}\to P$ as $i \to \infty$ for $T_i := N_i \log 2$, $i \in \N$, then
$$P=\cent(\widehat{Q})^\square.$$
The point of the first two parts of \cite[Proposition 5.5]{Hochman2010} is that a random translation of a fixed measure does satisfy these conditions. This yields the first claim.

For the latter statement, if $\mu$ CP generates $Q$, then by definition for $\mu$ typical $x$, we have that $\la \mu,x \ra_{N_i}\to Q$ for all sequences $N_i \to \infty$. Hence any accumulation point of the scenery distributions $\langle \mu \rangle_{x,T_i}$ must equal $\cent(\widehat{Q})^\square$, and we conclude that $\mu$ generates $\cent(\widehat{Q})^\square$, as claimed.
\end{proof}

In light of this proposition and the equivalence between FDs and CPDs given in Theorem \ref{thm:equivalence-CP-FD}, Theorem \ref{thm:measure-generating-FD-2} will be established once we prove the following proposition.

\begin{proposition}
\label{req:propo}
For any CP distribution $Q$ there exists a measure $\mu \in \cM$ which CP generates $Q$.
\end{proposition}

\begin{proof} Write
$$\cG := \{Q \in \CPD : \textrm{there exists $\mu \in \cM$ which CP-generates }Q\}.$$
We need to prove that $\cG = \CPD$. If $Q$ is ergodic, we know from Proposition \ref{prop:ergodic-CPD} that $Q$ almost all $\mu$ do generate $Q$. On the other hand, we have seen in Proposition \ref{prop:CPD-Poulsen} that ergodic CPDs are dense. Hence it is enough to show that $\cG$ is closed under the weak topology.

In the course of the proof we use notation from Section \ref{subsec:splicing}. Let $Q^i \in \cG$ and suppose there exists
$$Q = \lim_{i \to \infty} Q^i.$$
Since each $Q^i$ is a CP distribution, we know from Lemma \ref{lem:CP-distributions-are-closed} that $Q$ is a CP distribution, so in order to show that $Q \in \cG$ we are required to construct a measure $\mu$ which CP generates $Q$. For each $i \in \N$ let $\mu^i$ be a measure CP generating $Q^i$. Fix $0 < \eps < 1$ and choose a sequence $\epsilon_i \downarrow 0$ such that
$$\prod_{i=1}^\infty(1-\epsilon_i) = \eps.$$
Since $\mu^i$ generates $Q^i$ we can find $m_i \in \N$ such that $\mu^i(U_i)>1-\epsilon_i$ for the set
\[
U_i=\{x \in B_1 : d(\overline{\langle \mu^i,x \rangle}_{N},\overline{Q}^i) < \eps_i \textrm{ for every }N \geq m_i\}.
\]
We use the sequence $(m_i)$ to construct a sequence $(n_i)$ as follows: Let $n_1 = \max\{e^{m_1},e^{m_2}\}$ and for $i > 1$ put
\begin{eqnarray*}
k_i := \max\{e^{n_{i-1}},e^{m_i},e^{m_{i+1}}\} \quad \textrm{and} \quad n_i := m_i+k_i.
\end{eqnarray*}
We let $\spl = \spl_\n$ be the splicing map associated to the sequence $\mathbf{n}=(n_i)_{i \in \N}$; recall Section \ref{subsec:splicing}. Write
$$
\mu := \spl(\times_{i = 1}^\infty \mu^i) \quad \textrm{and} \quad U := \spl(\times_{i = 1}^\infty U_i).
$$
By the definition of product measure,
$$
\mu(U) = \lim_{M \to \infty} \prod_{i = 1}^M \mu^i(U_i) \geq \lim_{M \to \infty} \prod_{i = 1}^M (1-\epsilon_i) = \eps > 0.
$$
Hence $\mu_U$ is well-defined and
$$\MD(\mu_U,x) = \MD(\mu,x)$$
for $\mu$ almost every $x \in U$ by Proposition \ref{prop:MD-restriction}. Hence if we can prove for a fixed $z \in U$ that $\overline{\langle \mu,z \rangle}_N \to \overline{Q}$ as $N \to \infty$, then the normalized restriction $\mu_U$ CP generates $Q$ by definition. This is what we do in the Lemma \ref{lem:convergence} below.
\end{proof}

\begin{lemma}
\label{lem:convergence}
Given $z \in U$, we have
$$\lim_{N \to \infty}\overline{\langle \mu,z \rangle}_N = \overline{Q}.$$
\end{lemma}

\begin{proof}
Choose $x^i \in U_i$, $i \in \N$, such that $z = \spl(x^1,x^2,\dots)$. Define the sum
$$S_i = S_i(\n) = n_1+n_2+\dots+n_i.$$
For each $N \in \N$ choose $i = i(N) \in \N$ such that $S_i \leq N < S_{i+1}$. Notice that the sequence $(i(N))_{N \in \N}$ increases to infinity as the numbers $S_i$ increase to infinity. We will not write the dependence of $i$ on $N$ explicitly, but it is important to keep in mind that $i\to\infty$ as $N\to\infty$. Write the proportions:
$$p_N = \frac{S_i}{N} \quad \textrm{and} \quad q_N = 1-p_N.$$
Then $0 \leq p_N,q_N \leq 1$ for any $N \in \N$.

Fix a control parameter $K \in \N$, which we do not touch until the end of the proof after we have let $N \to \infty$. Assume $N$ is so large (depending on $K$) that $S_{i-1}+K < S_i$. This possible since $n_i \to \infty$ as $N \to \infty$.

We make use of the metric $d$ from Definition \ref{def:metriconmeasures}. The aim is to estimate for $d(\overline{\la \mu,z \ra}_N,\ol{Q})$ and for this we need to find suitable distribution decompositions of the difference $\overline{\la \mu,z \ra}_N-\ol{Q}$. We provide a different decomposition of this difference  depending on the position of $N-S_i$ with respect to $m_i$. In both cases we obtain a representation for the average $\overline{\la \mu,z \ra}_N-\overline{Q}$ and we see that when $N$ is very large, this representation is close to $0$. For this purpose, fix $f \in \Lip_1(X)$ and let $M := \|f\|_\infty$.
\begin{itemize}
\item[(1)] Suppose $N-S_i > m_{i+1}$. Then split the average $\overline{\la \mu,z \ra}_N$ into
$$p_N \Big(\frac{1}{S_i} \sum_{k = S_{i-1}}^{S_i-K}\delta_{\mu(\cdot|z_0^k)}\Big)+q_N \Big(\frac{1}{N-S_i} \sum_{k = S_i+K}^{N-1}\delta_{\mu(\cdot|z_0^k)}\Big)+\frac{1}{N}\sum_{\textrm{rest of }k } \delta_{\mu(\cdot|z_0^k)}$$
and denote this sum by $p_N F_N + q_N G_N + E_N$. Then as $p_N+q_N= 1$ we obtain
$$\int f \, \mathrm{d}(\overline{\la \mu,z \ra}_N-\overline{Q}) = p_N \int f \, \mathrm{d}(F_N-\ol{Q}) + q_N  \int f \, \mathrm{d}(G_N-\ol{Q}) + \int f \, \mathrm{d}E_N.$$
Moreover, we continue splitting
$$F_N-\ol{Q} = (F_N-\ol{Q}^i) + (\ol{Q}^i - \ol{Q}) \quad \textrm{and} \quad F_N-Q = (G_N-\ol{Q}^{i+1}) + (\ol{Q}^{i+1} - \ol{Q}).$$
Thus in this case we obtain an estimate for $\int f \, \mathrm{d}(\overline{\la \mu,z \ra}_N-\ol{Q})$ in the terms of
$$p_N d(F_N,\overline{Q}^{i}) +p_N d(\ol{Q}^i,\ol{Q}) +q_N d(G_N,\overline{Q}^{i+1})+q_N d(\ol{Q}^{i+1},\ol{Q})+ \int f \, \mathrm{d}E_N$$
as the first four terms can be estimated from above with their $d$-distance. We claim that
\begin{eqnarray}
\label{Fn} d(F_N,\overline{Q}^{i}) &\leq& \sqrt{d} 2^{-dK} + \eps_i, \\
\label{Gn} d(G_N,\overline{Q}^{i+1}) &\leq& \frac{S_i-K-S_{i-1}}{S_i} \cdot \sqrt{d} 2^{-dK} + \eps_{i+1},\\
\label{En}\int f \, \mathrm{d}E_N &\leq& M \cdot \frac{S_{i-1}+K}{S_i}.
\end{eqnarray}
In order not to interrupt the flow of the proof, these are shown in Lemma \ref{lem:case1} below.
\item[(2)] Suppose $N - S_i \leq m_{i+1}$. Then split the average $\overline{\la \mu,z \ra}_N$ into
$$\frac{1}{N}\sum_{k = S_{i-1}}^{S_i-K}\delta_{\mu(\cdot|z_0^k)}+\frac{1}{N}\sum_{\textrm{rest of }k } \delta_{\mu(\cdot|z_0^k)}  =: F_N' + E_N'.$$
In this case, we write
\begin{eqnarray*}\int f \,\mathrm{d}(\overline{\la \mu,z \ra}_N-\overline{Q}) &=& \int f \, \mathrm{d}(F_N'-\overline{Q}) + \int f \, \mathrm{d}E_N' \\
&=& \int f\, \mathrm{d}(F_N' - \ol{Q}^{i}) + \int f \, \mathrm{d}(\ol{Q}^{i} - \ol{Q}) + \int f \, \mathrm{d}E_N'.\end{eqnarray*}
Thus we obtain an estimate for $d(\overline{\la \mu,z \ra}_N,\ol{Q})$ in the terms of
$$d(F_N',\overline{Q}^{i}) +d(\ol{Q}^i,\ol{Q}) +\int f \, \mathrm{d}E_N'.$$
Now we have the estimates
\begin{eqnarray}
\label{Fndot}d(F_N',\overline{Q}^{i}) &\leq& \frac{S_i-K-S_{i-1}}{S_i}\cdot \sqrt{d} 2^{-dK}+ \eps_i,\\
\label{Endot}\int f \, \mathrm{d}E_N' &\leq&  M \cdot \frac{S_{i-1}+K+m_{i+1}}{S_i}.
\end{eqnarray}
The proof of these is deferred to Lemma \ref{lem:case2}  below.
\end{itemize}

We are left to analyze the estimates above. Notice that when $N \to \infty$, the fractions $S_{i-1}/S_i \to 0$ so we have
\[
\frac{S_i-K-S_{i-1}}{S_i} \to 1 \quad \textrm{and} \quad \eps_i \to 0\quad\textrm{as }N \to \infty.
\]
Thus we see that these upper bounds in (\ref{Fn}), (\ref{Gn}), and (\ref{Fndot}) tend to $0$ if we first let $N \to \infty$ and then $K \to \infty$. Moreover, by the choices of $n_i$, the numbers $m_{i+1}/S_i \to 0$ as $N \to \infty$ so the upper bounds in (\ref{En}) and (\ref{Endot}) tend to $0$ as $N \to \infty$ for any $K$. Since by our assumption $Q^i \to Q$ as $N \to \infty$, the terms $d(\ol{Q}^{i},\ol{Q}),d(\ol{Q}^{i+1},\ol{Q}) \to 0$ as $N \to \infty$, so the proof of the proposition is complete.
\end{proof}

We finish by proving the remaining estimates in the proof of Proposition \ref{lem:convergence}.

\begin{lemma}\label{lem:case1}
If $N-S_i > m_{i+1}$, then estimates (\ref{Fn}), (\ref{Gn}), and (\ref{En}) hold.
\end{lemma}

\begin{proof}
Recall that $z = \spl(x^1,x^2,\dots)$, so by the triangle inequality
$$d(F_N,\overline{Q}^{i}) \leq d(F_N,\overline{\langle \mu^i,x^i \rangle}_{S_i})+d(\overline{\langle \mu^i,x^i \rangle}_{S_i},\overline{Q}^{i})$$
and
\[
d(G_N,\overline{Q}^{i+1}) \leq d(G_N,\overline{\langle \mu^{i+1},x^{i+1} \rangle}_{N-S_i})+d(\overline{\langle \mu^{i+1},x^{i+1} \rangle}_{N-S_i},\overline{Q}^{i+1}).
\]
Since $z \in U$, we know that the points $x^{i} \in U_{i}$ and $x^{i+1} \in U_{i+1}$. Moreover, $N - S_i > m_{i+1}$ and $S_i > m_{i}$, so by the definitions of $U_{i}$ and $U_{i+1}$, we have
$$d(\overline{\langle \mu^{i},x^{i} \rangle}_{S_i},\overline{Q}^{i}) < \eps_i \quad \textrm{and} \quad d(\overline{\langle \mu^{i+1},x \rangle}_{N-S_i},\overline{Q}^{i+1}) < \eps_{i+1}.$$
Next, let us look at the term $d(G_N,\overline{\langle \mu^{i+1},x^{i+1} \rangle}_{N-S_i})$. Note that if $S_i\le k\le N-1$, then
\[
d(\mu(\cdot|x_0^k),\mu^{i+1}(\cdot|x_0^{k-S_i})) \le \sqrt{d}\cdot 2^{-K}.
\]
This follows from (\ref{eq:local-splicing}), just like in the proof Lemma \ref{lem:easy-convergence}. Using this we estimate
\begin{eqnarray*}
d(G_N,\overline{\langle \mu^{i+1},x^{i+1} \rangle}_{N-S_i}) &\leq& \frac{1}{N-S_i}\sum_{k = S_i}^{N-1} d(\mu(\cdot|x_0^k),\mu^{i+1}(\cdot|x_0^{k-S_i}))  \\
&\le& \frac{1}{N-S_i}\sum_{k = S_i}^{N-1} \sqrt{d} 2^{-dK} = \sqrt{d} 2^{-dK}.
\end{eqnarray*}
Moreover,
\begin{eqnarray*}
d(F_N,\overline{\langle \mu^i,x^i \rangle}_{S_i}) \leq \frac{1}{S_i}\sum_{k = S_{i-1}}^{S_i-K} d(\mu(\cdot|x_0^k),\mu^{i}(\cdot|x_0^{k-S_{i-1}})) \leq \frac{S_i-K-S_{i-1}}{S_i} \cdot \sqrt{d} 2^{-dK}.
\end{eqnarray*}
Moreover, for the distribution $E_N$, we see that the number of elements in its sum is exactly
$$N - [(S_i-K)-S_{i-1}] - [N-S_i] =  S_{i-1}+K,$$
and so by $N \geq S_i$ and $M = \|f\|_\infty$ we obtain the claim
\begin{eqnarray*}\int f \, \mathrm{d}E_N \leq M \cdot \frac{S_{i-1}+K}{N} \leq M \cdot \frac{S_{i-1}+K}{S_i}.\end{eqnarray*}
\end{proof}

\begin{lemma}\label{lem:case2} If $N - S_i \leq m_{i+1}$, then estimates (\ref{Fndot}) and (\ref{Endot}) hold.
\end{lemma}

\begin{proof}
A symmetric argument as in the proof of Lemma \ref{lem:case1} when estimating $G_N$ shows that
\begin{eqnarray*}d(F_N',\overline{Q}^{i}) \leq \frac{S_i-K-S_{i-1}}{N}\cdot \sqrt{d} 2^{-dK}+\eps_i \leq \frac{S_i-K-S_{i-1}}{S_i}\cdot \sqrt{d} 2^{-dK}+ \eps_i \end{eqnarray*}
as claimed. Moreover, the number of generations chosen in the sum over ``rest of $k$'' in $E_N'$ is exactly
$$S_{i-1}+K+(N-S_i) \leq S_{i-1}+K+m_{i+1}.$$
so by $N \geq S_i$ and $M = \|f\|_\infty$ we obtain the claim
\begin{eqnarray*}\int f \, \mathrm{d}E_N' \leq M \cdot \frac{S_{i-1}+K+m_{i+1}}{N} \leq M \cdot \frac{S_{i-1}+K+m_{i+1}}{S_i}.\end{eqnarray*}
\end{proof}

\section{Generic fractal distributions}
\label{sec:Baire-generic}

We now establish Theorem \ref{thm:Baire}, stated more precisely as follows.

\begin{theorem} \label{thm:Baire-2}
For a Baire generic measure $\mu\in\MM$ (where in $\MM$ we are considering the weak topology), it holds that
\[
 \TD(\mu,x) = \FD \quad\textrm{for }\mu\textrm{ almost all } x.
\]
\end{theorem}

As usual, we prove first an auxiliary result for CP distributions. This result is a consequence of the existence of measures that CP generate a given CPD.
\begin{proposition}
\label{prop:genericCP}
For a Baire generic $\mu \in \cM^\square$ the set of micromeasure distributions
$$\MD(\mu,x) \supset \CPD \quad \textrm{at $\mu$ almost every $x\in B_1$}.$$
In particular, if $\mu$ is supported on $B(0,1/2)$, then for Lebesgue almost every $\omega\in B(0,1/2)$ we have
$$\MD(\mu+\omega,x+\omega) = \CPD \quad \textrm{at $\mu$ almost every $x$}.$$
\end{proposition}

The idea of the proof is to choose a suitable countable dense subset $\cS \subset \CPD$ and prove that a Baire generic measure $\mu$ has $\cS$ as a subset of $\MD(\mu,x)$ at a $\mu$ typical $x$. Then the closedness of $\MD(\mu,x)$ will guarantee the claim. By the countability of $\cS$, we just need to verify the claim for a fixed $Q \in \cS$, as the countable intersection of Baire generic sets is Baire generic. This in turn can be obtained by proving that the property of being close to a measure $\mu$ whose CP scenery distribution $\la \mu,x\ra_{N}$ is weakly close to $Q$ in a set of large $\mu$ measure, is an open and dense property.

To deal with openness of measure theoretical properties using Euclidean balls, the dense subset of CPDs must consist of measures which give zero mass to all boundaries of dyadic cubes. This is guaranteed by the Lebesgue intensity properties of CPDs.

\begin{definition}We define
$$B_1^\circ := B_1 \setminus \bigcup_{N=1}^\infty \bigcup_{D \in \cD_N} \partial D.$$
That is, we remove the boundaries of dyadic cubes of all levels from $B_1$. Further, we let
\begin{eqnarray*}
\cF &= \{ \mu\in\cM^\square: \mu(B_1^\circ)=1\}.
\end{eqnarray*}
\end{definition}

\begin{lemma}
\label{lem:countablebaire}
There exists a countable dense subset $\cS \subset \CPD$ such that for any $Q \in \cS$ there exists a measure $\mu \in \cF$ that CP generates $Q$
\end{lemma}

\begin{proof} Since $\ECPD \subset \CPD$ is dense by the Poulsen property (Proposition \ref{prop:CPD-Poulsen}) and $\CPD$ is separable, there exists a countable subset $\cS \subset \ECPD$ which is dense in $\CPD$. Each CPD $Q$ has Lebesgue intensity measure, so $[Q](B_1^\circ) = 1$. This yields that $Q$ almost every $\mu$ gives mass $\mu(B_1^\circ) = 1$. On the other hand, by the ergodic theorem,  if $Q$ is ergodic, then $Q$ almost every $\mu$ CP generates $Q$, so we in particular fix one such $\mu$ with $\mu(B_1^\circ) = 1$. This yields the desired set $\cS$.
\end{proof}

We state a simple lemma which explains why we will work with measures in $\cF$.

\begin{lemma} \label{lem:continuity-no-mass-boundary}
For fixed $y\in B_1$ and $N\in\N$, the map  $\nu\to \overline{\la \nu,y\ra}_N$ is continuous at all $\mu\in\cF$ for which $\overline{\la \mu,y\ra}_N$ is defined.
\end{lemma}

\begin{proof}
The claim will follow if we can show that $\nu\to\nu(\cdot|y_0^j)$ is, for each fixed $j$, continuous at elements of $\cF$ for which $\nu(\cdot|y_0^j)$ is defined. Write $z=y_0^j$. Recall that by definition $\mu(\cdot | z) = T_{B_z}\mu_{B_z}$. Fix a sequence $(\mu_n)$ with $\mu_n \to \mu$. Since $\mu \in \cF$, we have $\mu(\partial B_z) = 0$, so by the weak convergence of measures (see for example Lemma \ref{lem:weak-convergence} applied to $f = \chi_{B_z}$) we have $\mu_n(B_z) \to \mu(B_z)$. Hence the measures $(\mu_n)_{B_z} \to \mu_{B_z}$ as $n \to \infty$. The homothety map $T_{B_z} : B_1 \to B_1$ is continuous, so the continuity follows.
\end{proof}

%Now, we cannot use the distance $d$ in $\cP(\cP(B_1))$ directly here due to technical issues, so we must stick to a modified metric
%\begin{eqnarray}\label{eq:modified-metric}
%d_M(P,Q) := \sup_{ f\in \Lip_1(X)} \Big\{ \int f\, \mathrm{d}(P-Q) : \|f\|_\infty \leq M\Big\}
%\end{eqnarray}
%for $P,Q \in \cP(\cP(B_1))$. This allows us a degree of freedom in the estimates in the end, but we still can recover weak convergence from this metric in the following sense: if $d_M(P_i,P) \to 0$ for all $M \in \N$, then the actual distance $d(P_i,P) \to 0$ as $i \to \infty$.

\begin{lemma}\label{lem:opengeneric}
Given any nonempty open set $\mathcal{O} \subset \cP(\MM^\square)$, $N \in \N$ and $c > 0$, let
\[
\cU :=  \bigl\{ \mu \in \cM^\square : \mu(\{x: \overline{\la \mu,x\ra}_N \in\mathcal{O} \}) > c \bigr\}.
\]
Then $\cF\cap\cU$ is contained in the interior of $\cU$.
\end{lemma}

\begin{proof}
For any $\nu\in\MM$, let
\[
A_\nu = \{x\in B_1^\circ: \overline{\la \nu,x\ra}_N \in\mathcal{O} \}.
\]
Fix $\mu\in\cF\cap\cU$, that is, $\mu(B_1^\circ)=1$ and $\mu(A_\mu)>c$.

We claim that there is an open set $\mathcal{V}  \ni\mu$ such $A_\mu\subset A_\nu$ for all $\nu\in\mathcal{V}$. Indeed, note that $\overline{\la \nu,x\ra}_N$ depends only on the dyadic cube of level $N$ which contains $x$. Now let $Z$ denote the (finite) collection of centers of dyadic cubes whose interior is contained in $A_\mu$. Then $A_\mu \subset A_\nu$ whenever
\[
 \nu \in \bigcap_{z\in Z} \{ \eta: \overline{\la \eta,z\ra}_N \in\mathcal{O} \},
\]
which contains a neighborhood of $\mu$ by Lemma \ref{lem:continuity-no-mass-boundary}. This is where we used the fact that $\mu\in\cF$.

Note that $A_\mu$ is open (as a union of open dyadic cubes). Since $\mu(A_\mu)>c$ and $\mu$ is Radon, there is a compact subset $K\subset A_\mu$ such that $\mu(K)>c$. If we let $\mathcal{W}=\{ \nu\in\MM: \nu(K)>c\}$, then $\mathcal{W}$ is open: indeed, if $\eta_n$ are in its complement and $\eta_n\to\eta$ then, using \cite[Theorem 1.24]{Mattila1995},
\[
\eta(K)\le \liminf_{n\to\infty} \eta_n(K)\le c,
\]
so the complement of $\mathcal{W}$ is closed.

We have seen that the set $\mathcal{V}\cap\mathcal{W}$ contains $\mu$, is open, and is contained in $\mathcal{U}$. Hence $\mu$ is in the interior of $\cU$, as claimed.

\end{proof}

\begin{lemma} \label{lem:densegeneric}
For any nonempty open sets $\mathcal{O} \subset\cP(\MM^\square)$ and $\cB\subset\cM^\square$, and any $\eps>0$, there exist $N_0\in\N$ and a measure $\mu\in\cB\cap\cF$ such that
\[
\mu(\{ x: \overline{\la \mu,x \ra}_N \in\mathcal{O} \textrm{ for all }N\geq N_0  \}) > 1-\eps.
\]
\end{lemma}
\begin{proof}
Recall the collection $\cS \subset \CPD$ from Lemma \ref{lem:countablebaire}. Since $\cS$ is dense, we can find a CP distribution $Q \in \cO \cap \cS$. By the definition of $\cS$, there exists a measure $\nu\in\cF$ that CP generates $Q$. Pick $\tau\in\cB$. Following the terminology of Hochman \cite[Section 8.2]{Hochman2010}, define the $(\nu,n)$-\textit{discretization} of $\tau$ by
\[
\tau_n = \sum_{D \in D_n}\tau(D) T_{D}^{-1} \nu.
\]
This is very similar to splicing, except that we use the measure $\tau$ for the first $n$ dyadic generations and the measure $\nu$ for all the others.
Since $\nu$ is a probability measure, we obtain $\tau_n \to \tau$ so there exists $n = n_{\delta} \in \N$ such that $\tau_n\in\cB$. Let $\mu$ be this discretization $\tau_n$.  Notice that $\mu\in\cF$ since $\nu\in\cF$.

Recall the symbolic coding of dyadic cubes introduced in Notations \ref{coding}. It follows from the definition of $\mu$ that for any $x\in B_1$ such that $\tau(x_0^n)>0$ and any $j\in \N$,
\begin{equation} \label{eq:splicing-cond-measures}
\mu(\cdot|x_0^{n+j}) = \nu(\cdot|x_n^{n+j}).
\end{equation}

We employ the metric on measures $d$ introduced in Section \ref{subsec:weak_convergence}. Pick $\delta>0$ such that $B_d(\overline{Q},2\delta)\subset \mathcal{O}$, where $B_d$ denotes the open ball in this metric. Since $\nu$ CP generates $Q$, there exists $m \in \N$ such that $\nu(U_\delta) > 1-\eps$ for
\[
U_\delta = \bigl\{y \in B_1: d(\overline{\la \nu,y \ra}_N,\ol Q)< \delta \textrm{ for all } N \geq m \bigr\}.
\]

Now if $N\ge n$, $f\colon \cP(\cM^\square)\to \R$ is $1$-Lipschitz, $\|f\|_\infty \leq 1$, $x\in B_1$ is such that $\nu(x_0^N)>0$, then  we deduce from (\ref{eq:splicing-cond-measures}) that
\begin{eqnarray*}
\biggl| \int f\,\mathrm{d}\overline{\la \mu,x \ra}_N - \int f\,\mathrm{d} \overline{\la \nu,x_n^\infty \ra}_N  \biggr| = \frac1N \biggl| \sum_{j=0}^{n-1} \int f\, \mathrm{d}\mu(\cdot|x_0^j) - \sum_{j=N}^{N+n-1}\int f \, \mathrm{d}\nu(\cdot|x_0^j) \biggr| \le \frac{2n\|f\|_\infty}{N} \leq \frac{2n}{N} .
\end{eqnarray*}
By taking $m$ larger if needed, we may further assume that $\frac{2n}{m}<\delta$. The above calculation then shows that
\[
d( \overline{\la \mu,x \ra}_N, \overline{\la \nu,x_n^\infty \ra}_N ) < \delta\quad\textrm{whenever }N\ge m.
\]
Hence, if we define
\[
V_\delta = \{ x\in B_1: x_n^\infty\in U_\delta\},
\]
we have that $\overline{\la \mu,x \ra}_N\in \mathcal{O}$ for all $N\ge m$. Furthermore, by the definition of $\mu$,
\[
\mu(V_\eps) = \sum_{D \in D_n} \tau(D) \nu(T_D V_\eps) \geq \sum_{D \in D_n} \tau(D) \nu(U_\eps) > 1-\eps.
\]
This concludes the proof.
\end{proof}

We can now conclude the proof of Proposition \ref{prop:genericCP}.

\begin{proof}[Proof of Proposition \ref{prop:genericCP}]
Let $\cS \subset \CPD$ be the countable dense subset from Lemma \ref{lem:countablebaire}. For $Q \in \cS$, $\eps > 0$ and $K \in \N$, we define
\[
  \cU_{Q,\eps,K} := \bigcup_{N \geq K} \textrm{interior}\left\{\mu \in \cM^\square : \mu(\{x : d(\overline{\la \mu,x \ra}_N, \overline{Q})  < \eps\}) > 1-\eps\right\}
\]
and
\[
  %\cR := \bigcap_{Q \in \cS} \bigcap_{\yli{\eps \in \Q}{0 < \eps < 1}}  \bigcap_{K \in \N}  \cU_{Q,\eps,K}.
  \cR := \bigcap_{Q \in \cS} \bigcap_{\eps \in \Q, \, 0 < \eps < 1}  \bigcap_{K \in \N}  \cU_{Q,\eps,K}.
\]
The collection $\cU_{Q,\eps,K}$ is open as a union of open sets over $N \geq K$. Moreover, by Lemma \ref{lem:opengeneric} applied to the open set $\cO = B_d(\overline{Q},\eps) \subset \cP(\MM^\square)$, the collection $\cU_{Q,\eps,K}$ contains the set
\[
\mathcal{D}_{Q,\eps,K}= \bigcup_{N\ge K} \left\{ \mu\in\cF: \mu\bigl(\{ x : d(\overline{\la \mu,x \ra}_N, \overline{Q})  < \eps\}\bigr) > 1-\eps\right\}.
\]
By Lemma \ref{lem:densegeneric} the set $\mathcal{D}_{Q,\eps,K}$ is dense: for a given nonempty open set $\cB \subset \cM^\square$, we can choose $\mu \in \cF\cap\cB$ such that $d(\overline{\la \mu,x \ra}_N, \overline{Q})  < \eps$ happens for all $N \geq N_0$ in a set of $\mu$ measure $>1-\eps$. In particular, we can find $N \geq K$ with this property, showing that $\cB$ meets $\mathcal{D}_{Q,\eps,K}$. Hence $\mathcal{D}_{Q,\eps,K}$ is dense as claimed, and so is $\cU_{Q,\eps,K}$.

Since the set $\cR$ is a countable intersection of sets with dense interiors, its complement is a set of first category. Fix $\mu \in \cR$ and let  $\{Q_j\}$ be an enumeration of $\cS$. Let $\eps_i \searrow 0$ be a sequence of rational numbers. Then for some $N_{i,j} \nearrow \infty$ we obtain $\mu(A_{i,j}) \geq 1-\eps_i$ for
\[
A_{i,j} := \{x \in B_1: d(\overline{\la \mu,x \ra}_{N_{i,j}}, \overline{Q}_j)  <  \eps_i\}.
\]

Write
\[
A = \bigcap_{j\in \N} \limsup_{i \to \infty} A_{i,j}.
\]
Then as $\eps_i \searrow 0$, we have by the convergence of measures that $\mu(A) = 1$. Fix $x \in A$. For each $j$ there are infinitely many $i$ such that
\[
d(\overline{\la \mu,x \ra}_{N_{i,j}}, \overline{Q}_j)  <  \eps_{i}.
\]
This shows that $\cS \subset \MD(\mu,x)$. Since $\MD(\mu,x)$ is closed in $\cP(\Xi)$ we have $ \CPD \subset \MD(\mu,x)$ at $\mu$ almost every $x$.

The second statement is immediate from the fact that $\MD(\mu+\omega,x+\omega) \subset \CPD$ for Lebesgue almost all $\omega$ (recall Proposition \ref{thm:micro-distris-are-CPDs}).
\end{proof}

\begin{proof}[Proof of Theorem \ref{thm:Baire-2}]
Given $Q\in\cS, \eps>0, K\in\N$, let
\[
  \widetilde{\cU}_{Q,\eps,K} = \textrm{interior}\{ \mu\in\cM: \mu^\square\in \cU_{Q,\eps,K} \},
\]
where $\cU_{Q,\eps,K}$ is as in the proof of Proposition \ref{prop:genericCP}. These sets are open by definition. We claim they are also dense.

Let $\pi\colon \cM\to\cM^\square$ be the map $\mu\to\mu^\square$. This map is not continuous (and is not everywhere defined), but it is defined and continuous on the set $\cD\subset\cM$ of measures which give zero mass to the boundary of $B_1$ and positive mass to $B_1$. It is easy to see that $\cD$ is in fact dense, and $\cD^\square$ is dense in $\cM^\square$ (where, as usual, $\cD^\square=\{\mu^\square:\mu\in\cD\}$). Hence, since $\cU_{Q,\eps,K}$ is open and dense, $\cD^\square\cap \cU_{Q,\eps,K}$ is dense in $\cM^\square$. Since measures in $\cD$ are continuity points of $\pi$, the set $\pi^{-1}\left(\cD^\square\cap \cU_{Q,\eps,K}\right)$ is contained in the interior of $\pi^{-1}(\cU_{Q,\eps,K})$. We will see that it is also dense in $\cM$.

%\[
%D(\mu,\nu) = \sum_{N=1}^\infty 2^{-N} \min\left(1, \sup\left\{ \int f\,d(\mu-\nu) : f\in\Lip(\R^d), \textrm{supp}(f)\subset B(0,N)  \right\} \right).
%\]
%This is a variant on the metric in \cite[Remark 14.15]{Mattila1995}.

Let $D$ be a metric on $\cM$ inducing the weak topology (see e.g. \cite[Remark 14.15]{Mattila1995} for an instance of such a metric). Fix $\tau\in\cM$ and $\e>0$. Pick $\nu\in\cD$ such that $D(\nu,\tau)<\e/2$. Since $\nu\in\cD$, $\varrho:=\nu(B_1)>0$ so in particular $\nu^\square$ is well defined. Now pick a sequence $\mu_n\to\nu^\square$ with $\mu_n\in  \cD^\square\cap \cU_{Q,\eps,K}$; this is possible by density. Even though $\mu_n$ is a measure on $B_1$, we identify it with a measure on $\R^d$. Let
\[
\nu_n= \nu|_{\R^d\setminus B_1} + \varrho\,\mu_n.
\]
Then it is easy to see that $\nu_n\to\nu$ weakly, so we have $D(\nu_n,\tau)<\e$ for some $n$. By construction $\nu_n^\square=\mu_n$ and therefore $\nu_n\in\pi^{-1}\left(\cD^\square\cap \cU_{Q,\eps,K}\right)$, showing that this set, and hence also $\widetilde{\cU}_{Q,\eps,K}$, is dense as claimed.

Now, similar to the proof of Proposition \ref{prop:genericCP}, define
\[
  %\cR_{\mathbf{0}} =  \bigcap_{Q \in \cS} \bigcap_{\yli{\eps \in \Q}{0 < \eps < 1}}  \bigcap_{K \in \N}  \widetilde{\cU}_{Q,\eps,K}.
  \cR_{\mathbf{0}} =  \bigcap_{Q \in \cS} \bigcap_{\eps \in \Q, \, 0 < \eps < 1}  \bigcap_{K \in \N}  \widetilde{\cU}_{Q,\eps,K}.
\]
Fix $P \in \FD$. Theorem \ref{thm:equivalence-CP-FD} implies that there exists a CP distribution $Q$ such that
\[
P = \cent(\widehat{Q})^\square.
\]
Since $\mu \in \cR_{\mathbf{0}}$, the proof of Proposition \ref{prop:genericCP} implies that $\CPD \subset \MD(\mu,x)$ at $\mu$ almost every $x \in B_1$, so in particular $Q \in \MD(\mu,x)$ at these $x$. Then, by Proposition \ref{CPtoFD}, we know that $P \in \TD(\mu,x)$ at $\mu$ almost every $x\in B_1$. We also know that at $\mu$ almost every $x\in B_1$ we have $\TD(\mu,x) \subset \FD$ by Theorem \ref{thm:TDs-are-FDs}. We have shown that $\TD(\mu,x)=\FD$ for any $\mu\in\cR_{\mathbf{0}}$ and $\mu$ almost all $x\in B_1$.

To finish the proof, define
\[
\cR = \bigcap_{\mathbf{n}\in\Z^d} \left\{ T_{\mathbf{n}}\mu: \mu\in\cR_{\mathbf{0}}  \right\},
\]
i.e.\ we intersect all integer translates of $\cR_{\mathbf{0}}$. This set is a countable intersection of open and dense sets (notice that for fixed $\mathbf{n}$, the map $T_{\mathbf{n}}$ is a homeomorphism of $\cM$) and thus its complement is a set of first category. Moreover, if $\mu\in \cR$, then for all integer vectors $\mathbf{n}$, $\TD(\mu,x)=\FD$ for $\mu$ almost all $x\in B(\mathbf{n},1)$. As the union of these balls covers $\R^d$, we are done.
\end{proof}

\appendix

\section{Remarks on the choice of norm}

We have so far followed \cite{Hochman2010} in working with the $L^\infty$ norm of $\R^d$. However, other than simplifying some proofs, there is nothing special about this norm, and all the results from \cite{Hochman2010} that we require work equally well with any other norm. Since in geometric measure theory one uses mainly the Euclidean norm, our geometric applications in \cite{KSS2013} are simplified if we use Euclidean versions of the results presented here.

Rather than verifying all the proofs from \cite{Hochman2010} can be made to work with any norm, we will explain how to deduce the results from their $L^\infty$ version. This is most clear for the extended version of FDs, since in this case choosing different norms amounts to only to a different choice of normalization; see the discussion in \cite[Section 3.1]{Hochman2010}. For restricted versions things also work, helped by the uniqueness of the extended version.

For the sake of completeness, we provide details of the modifications needed to carry the theory to the setting of an arbitrary norm. Note that the concepts of $\FD$ and $\TD$ are effectively tied to the $L^\infty$ norm $\|\cdot\|$ which was implicit in their definition. Let $\|\cdot\|'$ be any other norm on $\R^d$. We will use an apostrophe to denote the corresponding concepts defined in terms of this new norm. For example, $B'_1$ is the unit ball in the norm $\|\cdot\|'$, the set of fractal distributions with respect to this norm is $\FD'$, etc. Further, for $\mu\in\cM$ with $0\in\supp\mu$ we will denote $\mu'=\frac{1}{\mu(B'_1)}\mu$ and $\mu^\Diamond= \mu_{B'_1}$. As with $\mu^*, \mu^\square$, we use the same notation to indicate postcomposition with these maps, as in $Q', Q^\Diamond$, etc.

The proof of the one-one correspondence $Q\to Q^\Diamond$ between restricted and extended versions (\cite[Lemma 3.1]{Hochman2010}) is independent of the chosen norm. Thus we will assume this.

\begin{proposition}
 \begin{enumerate}
  \item The map $\mu\to \mu'$ is an isomorphism (that is, a bijective factor map) from $(\cM^*,S^*)$ to $(\cM', S')$.
  \item If $Q\in\FD$ then $Q'\in\FD'$ and conversely for any $P\in\FD'$ there is $Q\in\FD$ with $Q'=P$.
  \item Theorem \ref{thm:TDs-are-FDs} continues to hold with the norm $\|\cdot\|'$, that is, if $\mu\in\cM$, then for $\mu$ almost all $x$, $\TD'(\mu,x)\subset \FD'$.
  \item Theorem \ref{thm:ergodic-components-of-FDs-are-FDs} holds for the norm $\|\cdot\|'$, that is, if $Q\in\FD'$ then the ergodic components of $Q$ are also in $\FD'$.
  \item Theorem \ref{thm:equivalence-CP-FD} holds for the norm $\|\cdot\|'$, that is, if $Q$ is a CPD then $\cent'(Q)\in\FD'$ and, conversely, given any $P\in\FD'$, there is an extended CPD $Q$ with $\cent'(Q)=P$. Here $\cent'(Q)$ is the push-down of $Q\times\lambda$ under $((\mu,x),t)\to S'_t T_x\mu$ (where $\lambda$ is normalized Lebesgue measure on $[0,\log 2)$).
 \end{enumerate}
\end{proposition}
\begin{proof}
\begin{enumerate}
 \item This is routine from the definitions.
 \item Given the first part, we only need to show that if $Q\in\FD$ then $Q'$ is quasi-Palm. Suppose $Q'(Y)=1$. Since $\mu\to\mu'$ is a bijection, $Y=Z'$ where $Q(Z)=1$. By the quasi-Palm property of $Q$, $\mu_{x,t}\in Z$ for $Q$ almost all $\mu$ almost all $x$. But then $T_x\mu'\in Z'$ for $Q$ almost all $\mu$ and $\mu$ almost all $x$, and whence also for $Q'$ almost all $\mu'$ and $\mu'$ almost all $x$, since $\mu$ and $\mu'$ are equivalent.
  \item Fix $\mu\in\cM$, and let $x$ be a point such that $\TD(\mu,x)\subset\FD$; by Theorem \ref{thm:TDs-are-FDs} we know this happens for $\mu$ almost all $x$. 
  
  Assume first that the unit ball of $\|\cdot\|'$ is contained in $B_1$. In this case, we have $S_t^\Diamond T_x \mu= (S_t^\square T_x \mu)^\Diamond$. Suppose $\la \mu \ra'_{x,T_j}\to Q$; by passing to a subsequence we may assume that $\la \mu \ra_{x,T_j}\to P$, where $P\in\FD$ by our choice of $x$. Then, writing $\widehat{P}$ for the extended version of $P$, we have $(\widehat{P}')^\Diamond=Q$; by the previous part, $\widehat{P}'$ is in $\FD'$ and by uniqueness is therefore the extended version of $Q$, showing that $Q\in\FD'$ as claimed. Note that this argument works in reverse: if $\TD'(\mu,x)\in \FD'$, then $\TD(\mu,x)\in\FD$.

   The general case now follows by considering first the norm $\|\cdot\|''=\max(\|\cdot\|,\|\cdot\|')$ (whose unit ball is contained in $B_1$) and then using the result for this norm to deduce the same for $\|\cdot\|'$ (which has a larger unit ball).
 \item  This assertion follows from the first two: we know from Theorem \ref{thm:ergodic-components-of-FDs-are-FDs} that if $Q= \int Q_\mu \mathrm{d}Q(\mu)$ is the ergodic decomposition of $Q$, then $Q_\mu\in\FD$ for $Q$ almost all $\mu$. By the first part, the ergodic decomposition of $Q'$ is $\int Q'_\mu \mathrm{d}Q'(\mu)$, and by the second part $Q'_\mu\in\FD'$ for $Q'$ almost all $\mu$.
 \item Note that $\cent'(Q)=(\cent(Q))'$ so both statements are consequence of the result for the $L^\infty$ norm and the first part.
\end{enumerate}
\end{proof}

Now it becomes clear that Theorems \ref{thm:FDs-are-closed-2}, \ref{thm:Poulsen-2}, \ref{thm:measure-generating-FD-2} and \ref{thm:Baire-2} hold for any norm and in particular for the Euclidean norm. This can be deduced either from the fact that our proofs are norm-independent (once we have norm-independent versions of the  results we assume), or by following simple arguments of the kind above to pass from the $L^\infty$ norm to an arbitrary norm.

\section*{Acknowledgments} We thank M. Hochman for many enlightening discussions on the topics related to this paper.

\bibliographystyle{abbrv} % or amsplain?
\bibliography{fractal_distributions-USM_FINAL}

\end{document}